\documentclass[a4paper,11pt]{article}

\usepackage{a4wide}
\usepackage{graphicx}
\usepackage{amssymb,amsmath,amsthm,mathrsfs,xfrac,mathtools,}
\usepackage{enumerate, enumitem}
\usepackage[title]{appendix}
\usepackage{amsfonts}
\usepackage{hyperref}
\usepackage[utf8]{inputenc}
\usepackage{color}
\usepackage{pgf,tikz}
\usetikzlibrary{arrows}

\newtheorem{proposition}{Proposition}[section]
\newtheorem{theorem}[proposition]{Theorem}
\newtheorem{lemma}[proposition]{Lemma}
\newtheorem{definition}[proposition]{Definition}

\newtheorem{remark}[proposition]{Remark}
\newtheorem{lgrthm}[proposition]{Algorithm}

\numberwithin{equation}{section}
\allowdisplaybreaks[4]

\renewenvironment{proof}{\smallskip\noindent\emph{\textbf{Proof.}}%
  \hspace{1pt}}{\hspace{-5pt}{\nobreak\quad\nobreak\hfill\nobreak%
    $\square$\vspace{2pt}\par}\smallskip\goodbreak}

\newcommand{\C}[1]{\mathbf{C^{#1}}}
\newcommand{\Cc}[1]{\mathbf{C_c^{#1}}}
\renewcommand{\L}[1]{\mathbf{L^#1}}
\newcommand{\Lloc}[1]{\mathbf{L^{#1}_{\mathbf{loc}}}}
\newcommand{\BV}{\mathbf{BV}}

\newcommand{\modulo}[1]{{\left|#1\right|}}
\newcommand{\norma}[1]{{\left\|#1\right\|}}
\newcommand{\caratt}[1]{{\chi_{\strut#1}}}
\newcommand{\reali}{{\mathbb{R}}}
\newcommand{\R}{\mathbb R}

\newcommand{\N}{{\mathbb N}}
\newcommand{\M}{{\mathcal{M}}}
\newcommand{\interi}{{\mathbb{Z}}}

\renewcommand{\epsilon}{\varepsilon}
\renewcommand{\phi}{\varphi}

\renewcommand{\theta}{\vartheta}

\newcommand{\tv}{\mathinner{\rm TV}}

\renewcommand{\d}[1]{\mathinner{\mathrm{d}{#1}}}
\newcommand{\NORM}[1]{{|\hspace{-1pt}\|{#1}\|\hspace{-1pt}|}}

\newcommand{\brho}{\boldsymbol{\rho}}
\newcommand{\bsigma}{\boldsymbol{\sigma{}}}
\newcommand{\bv}{\boldsymbol{v}}

\newcommand{\dt}{{\Delta t}}
\newcommand{\dx}{{\Delta x}}

\newcommand{\rh}[1]{\rho^{n}_{#1}}

\DeclareMathOperator{\sgn}{sgn}

\newcommand{\vl}[1]{v_{\ell,#1}}
\newcommand{\vr}[1]{v_{r,#1}}
\newcommand{\se}[1]{S_{\ell,#1}}
\newcommand{\sr}[1]{S_{r,#1}}

\newcommand{\be}{\begin{equation}}
\newcommand{\ee}{\end{equation}}

\definecolor{ffqqqq}{rgb}{1.,0.,0.}
\definecolor{uuuuuu}{rgb}{0.26666666666666666,0.26666666666666666,0.26666666666666666}

\makeatletter
\let\@fnsymbol\@arabic
\makeatother

\makeatletter
\newcommand\appendix@section[1]{%
\refstepcounter{section}%
\orig@section*{Appendix \@Alph\c@section: #1}%
\addcontentsline{toc}{section}{Appendix \@Alph\c@section: #1}%
}
\let\orig@section\section
\g@addto@macro\appendix{\let\section\appendix@section}
\makeatother

\title{A multi-lane macroscopic traffic flow model for simple networks}

\author{Paola Goatin\footnotemark[1] \and Elena Rossi\footnotemark[1]}

\date{}

\begin{document}
\maketitle
\footnotetext[1]{Inria Sophia Antipolis - M\'editerran\'ee,
  Universit\'e C\^ote d'Azur, Inria, CNRS, LJAD, 2004 route des
  Lucioles - BP 93, 06902 Sophia Antipolis Cedex, France. E-mail:
  \texttt{\{paola.goatin, elena.rossi\}@inria.fr}}

\begin{abstract}

  \noindent
  We prove the well-posedness of a system of balance laws inspired
  by~\cite{HoldenRisebro}, describing macroscopically the traffic flow
  on a multi-lane road network. Motivated by real applications, we
  allow for the the presence of space discontinuities both in the
  speed law and in the number of lanes. This allows to describe a
  number of realistic situations.  Existence of solutions follows from
  compactness results on a sequence of Godunov's approximations, while
  $\L1$-stability is obtained by the doubling of variables technique.
  Some numerical simulations illustrate the behaviour of solutions in
  sample cases.

  \medskip

  \noindent\textit{2010~Mathematics Subject Classification:} 35L65,
  90B20, 82B21.

  \medskip

  \noindent\textit{Keywords:} macroscopic multi-lane traffic flow
  model on networks; Godunov scheme; well-posedness.
\end{abstract}


\section{Introduction}

Macroscopic traffic flow models consisting of hyperbolic balance laws
have been developed in the scientific literature starting from the
celebrated Lighthill-Whitham-Richards (LWR)
model~\cite{LW1955,Richards1956}. Despite its simplicity, the LWR
model is able to capture the basic features of road traffic dynamics,
such as congestion formation and propagation. Nevertheless, it cannot
describe many aspects of road traffic complexity. To this end, several
improved models accounting for specific flow characteristics have
subsequently been introduced: second-order models accounting for a
momentum equation (see e.g.~\cite{AwRascle}), multi-population models
distinguishing between different classes of vehicles
(e.g.~\cite{BenzoniColombo}), etc.

In this paper, we are interested in describing carefully the traffic
dynamics on road networks with several lanes, allowing for lane change
and overtaking. Multi-lane models for vehicular traffic have been
proposed in~\cite{CC2006,HoldenRisebro,KW1,KW2}. In the macroscopic setting, these models consist in a system of balance laws in which the transport is expressed by a LWR equation for each lane, and the source term accounts for the lane change rate. In particular, the equations of the system are coupled in the source term only. \\
Aiming to describe realistic situations in detail, we allow for the
speed laws and the number of lane to change along the road. In the
study, for sake of simplicity, we consider the model proposed
in~\cite{HoldenRisebro}, but more general source terms could be
taken into account.

We consider an infinite road described by the real line.  Let
$\M_\ell\subset \N^+$ be the set of indexes of the {\it active} lanes
on $]-\infty, 0[$, with $M_\ell := \modulo{\M_\ell}\geq 1$ its
cardinality, and $\M_r\subset \N^+$ be the set of indexes of the {\it
  active} lanes on $]0,+\infty[$, with $M_r := \modulo{\M_r}\geq 1$.
Let us consider $M\geq \max\{M_\ell, M_r\}$, its choice depending on
the specific situation under study.

To cast the problem in a general setting, we extend the road
considering the same number of lanes $M$ on the left and on the right
of $x=0$. More precisely, we assume that there are $M - M_\ell$ and
$M - M_r$ additional empty lanes on $]-\infty,0[$, respectively
$]0,+\infty[$.  Moreover, we prevent vehicles from passing from the
active to the fictive lanes added, see condition~\eqref{eq:sghost}
below. In the same way, we can consider multiple separate roads, thus
accounting for network nodes.

The problem under consideration is then the following: for
$x \in \reali$ and $t>0$, the vehicle density $\rho_j=\rho_j(t,x)$ on
lane $j$ solves the Cauchy problem
\begin{equation}
  \label{eq:M}
  \left\{
    \begin{array}{l@{\quad}l}
      \partial_t \rho_j + \partial_x f_j (x,\rho_j)
      = S_{j-1}(x,\rho_{j-1}, \rho_j) - S_j (x, \rho_j, \rho_{j+1})
      & j=1, \ldots, M,
      \\
      \rho_j(0,x) = \rho_{o,j} (x)
        & j=1,\ldots, M,
    \end{array}
  \right.
\end{equation}
with
\begin{align}
  \label{eq:vj}
  v_j (x,u) = \
  & H (x) \, \vr{j} (u) + (1 - H (x)) \, \vl{j} (u),
  \\
  \label{eq:flr}
  f_{\ell,j}(u) = \
  & u \, \vl{j} (u), \quad
    f_{r,j} (u) = \
    u \, \vr{j} (u),
  \\
  \label{eq:3}
  f_j(x,u) = \
  & u \, v_j (x,u) 
    =\ H (x) f_{r,j} (u) + (1- H (x)) f_{\ell,j} (u),
\end{align}
for $j=1, \ldots, M$, where $H$ is the Heaviside function.  The
velocities $v_{d,j}$, for $d= \ell, r$ and $j=1,\ldots, M$, are
strictly decreasing positive functions such that $v_{d,j} (1) = 0$. We
assume that each map $f_{d,j} (u) = u \, v_{d,j} (u)$ admits a unique
global maximum in the interval $[0,1]$, attained at $u=\theta^j_d$. We
set
\begin{equation}
  \label{eq:theta}
  \theta^j (x) = H(x) \,  \theta^j_r+ (1- H(x))\, \theta^j_\ell.
\end{equation}
Moreover, we set $\rho_{o,j}:\R\to [0,1]$ for $j=1,\ldots, M,$ and
\begin{align}
  \label{eq:idM}
  &\rho_{o,j}(x) = 0 \quad \mbox{ for }  x\in\, ]-\infty,0[ \mbox{ and } j\not\in\M_\ell,\\
  \label{eq:idM1}
  &\rho_{o,j}(x) = 1 \quad \mbox{ for }  x\in\, ]0,+\infty[ \mbox{ and } j\not\in\M_r.
\end{align}
Concerning the source terms, accounting for the flow rate across
lanes, we define, as in \cite{HoldenRisebro},
\begin{equation}
  \label{eq:2}
  \begin{aligned}
    S_{d,j} (\rho_j, \rho_{j+1}) = \ & \left[ \left( v_{d,j+1}
        (\rho_{j+1}) - v_{d,j} (\rho_j) \right) ^+ \rho_j - \left(
        v_{d,j+1} (\rho_{j+1}) - v_{d,j} (\rho_j) \right) ^-
      \rho_{j+1} \right]
    \\
    = \ & \left( v_{d,j+1} (\rho_{j+1}) - v_{d,j} (\rho_j) \right)
    \left\{
      \begin{array}{l@{\quad}l}
        \rho_j &  v_{d,j+1} (\rho_{j+1}) \geq v_{d,j} (\rho_j) ,
        \\
        \rho_{j+1} &  v_{d,j+1} (\rho_{j+1}) < v_{d,j} (\rho_j) ,
      \end{array}
    \right.
  \end{aligned}
\end{equation}
for $d=\ell,r$ and $j=1, \ldots, M-1$, where
$(a)^+=\max\left\{a,0\right\}$ and $a^-=-\min\{a,0\}$. To account for
separate lanes, such as different roads or fictive lanes, we set
\begin{equation}
  \label{eq:sghost}
  S_{d,j_d} (u,w) = 0 \qquad \mbox{ for some } j_d \in \left\{1,\ldots,M-1\right\},~ d=\ell,r.
\end{equation}
The functions appearing in the source term are then defined as follows
\begin{align}
  \label{eq:4}
  S_j (x,u,w) = \
  &
    H (x) \, \sr{j} (u,w) + (1 - H (x)) \, \se{j} (u,w)
  &
    \mbox{ for } j= \
  & 1, \ldots, M-1,
  \\
  \label{eq:sbordo}
  S_0 (x,u,w) = \
  & S_M (x,u,w) = 0.
  &
\end{align}
For the sake of shortness, introduce the notation
$\brho = (\rho_1, \ldots, \rho_M)$, so that the initial data
associated to problem~\eqref{eq:M}--\eqref{eq:idM}--\eqref{eq:idM1}
read $\brho(0,x) = \brho_o (x)$.

  \begin{remark}
    For simplicity, and with slight abuse of notation, we consider
    $\brho=\brho(t,x)$ for $t>0$, $x\in\, \reali$.  However, we
    will show that, by~\eqref{eq:idM}, \eqref{eq:idM1}
    and~\eqref{eq:sghost}, there holds $\rho_j(t,x) = 0$ for all
    $t>0$, $x\in\, ]-\infty,0[$ and $j\not\in\M_\ell$, respectively
    $\rho_j(t,x) = 1$ for all $t>0$, $x\in\, ]0,+\infty[$ and
    $j\not\in\M_r$.
  \end{remark}

  \medskip

  Following~\cite[Definition~5.1]{KRT2003}, see
  also~\cite[Definition~2.1 and Formula~(5.8)]{KRT2002}
  and~\cite[\S~8.3]{HoldenRisebroBook2015}, we recall the definition
  of \emph{weak entropy solution} for~\eqref{eq:M}--\eqref{eq:idM}--\eqref{eq:idM1}.
  \begin{definition}
    \label{def:sol}
    A map
    $\brho= (\rho_1,\ldots,\rho_M) \in \L\infty ([0,T]\times \reali;
    [0,1]^M) $ is a \emph{weak entropy solution} to the initial value
    problem~\eqref{eq:M} if
    \begin{enumerate}
    \item\label{it:weak} for any
      $\phi \in \Cc1 ([0,T[ \times \reali; \reali)$ and for all
      $j=1,\ldots, M$,
      \begin{align*}
        \int_0^T\int_{\reali} \left(\rho_j \, \partial_t \phi
        + f_j (x,\rho_j) \, \partial_x \phi
        + \left(S_{j-1} (x,\rho_{j-1}, \rho_j) - S_j (x, \rho_j, \rho_{j+1})\right) \phi
        \right) \d{x} \d{t} &
        \\
        + \int_\reali \rho_{o,j} \, \phi (0,x)  \d{x} & = 0.
      \end{align*}
    \item
      \label{it:entropy} for any
      $\phi \in \Cc1 ([0,T[ \times \reali; \reali^+)$, for any
      $c\in[0,1]$ and for all $j=1,\ldots, M$
      \begin{align*}
        \int_0^T\int_{\reali} \left\{\modulo{\rho_j-c} \, \partial_t \phi
        + \sgn (\rho_j - c) \left(f_j (x,\rho_j) - f_j (x, c)\right) \partial_x \phi
        \right.
        &
        \\
        \left. \qquad
        + \sgn (\rho_j - c)
        \left(S_{j-1} (x,\rho_{j-1}, \rho_j) - S_j (x, \rho_j, \rho_{j+1})\right) \phi
        \right\} \d{x} \d{t}&
        \\
        + \int_0^T\modulo{f_{r,j} (c) - f_{\ell,j} (c)} \, \phi (t,0) \d{t}
        + \int_\reali \modulo{\rho_{o,j} -c} \phi (0,x)  \d{x}
        & \geq\  0.
      \end{align*}
    \end{enumerate}
  \end{definition}

  The rest of the paper is organised as follows. In
  Section~\ref{sec:existence} we construct a sequence of approximate
  solutions based on Godunov finite volume scheme and we prove its
  convergence towards a solution of~\eqref{eq:M}. We then provide a
  $\L1$-stability estimate with respect to the initial data, which
  implies the uniqueness of solutions. Specific situations and the
  corresponding numerical simulations are discussed in
  Section~\ref{sec:num}.

  \section{Well-posedness}
  \label{sec:existence}

  We define the map $\bv: [0,1] \to \reali^{2M}$ by setting
  $\bv_j = \vl{j}$ and $ \bv_{M+j}= \vr{j}$, for $j=1, \ldots, M$.
  Moreover we define
  \begin{equation}
    \label{eq:normav}
    \begin{aligned}
      V_{\max} = \ & \norma{\bv}_{\C0([0,1];\reali^{2 M})} =
      \max_{\substack{j=1,\dots,M\\d=\ell,r}}
      \norma{v_{d,j}}_{\L\infty ([0,1];\reali)},
      \\
      \mathcal{V} = \ & \norma{\bv}_{\C1([0,1];\reali^{2M})} =
      \max_{\substack{j=1,\dots,M\\d=\ell,r}}
      \norma{v_{d,j}}_{\L\infty ([0,1];\reali)} +
      \max_{\substack{j=1,\dots,M\\d=\ell,r}}
      \norma{v'_{d,j}}_{\L\infty ([0,1];\reali)}.
    \end{aligned}
  \end{equation}
  We introduce the following quantity, which corresponds to the
  $\L1$--norm of the vector $\brho$ computed on \emph{active} lanes:
  \begin{equation}
    \label{eq:norma1}
    \NORM{\brho} =
    \sum_{j\in\M_\ell} \norma{\rho_j}_{\L1 (]-\infty,0[)}
    +
    \sum_{j\in\M_r} \norma{\rho_j}_{\L1 (]0,+\infty[)}.
  \end{equation}

  \medskip

  Introduce a uniform space mesh of width $\dx$ and a time step $\dt$,
  subject to a CFL condition, to be detailed later on. For
  $k\in \interi$ set
  \begin{align*}
    x_k = \
    & \left(k+\frac12\right) \dx,
    &
      x_{k - 1/2} = \
    & k \dx,
  \end{align*}
  where $x_k$ denotes the centre of the cell, while $x_{k\pm 1/2}$ its
  interfaces. Observe that $x=0$ corresponds to $x_{-1/2}$, so that
  non negative integers denote the cells on the positive part of the
  $x$-axis. Set $N_T= \lfloor T/\dt \rfloor$ and let $t^n= n \, \dt$,
  for $n=0, \ldots, N_T$, be the time mesh. Set $\lambda = \dt / \dx$.
  Approximate the initial data in the following way: for
  $j=1,\ldots,M$, for $k\in \interi$
  \begin{displaymath}
    \rho_{j,k}^0 = \frac1\dx \int_{x_{k-1/2}}^{x_{k+1/2}}\rho_{o,j} (x) \d{x}.
  \end{displaymath}
  Define a piece-wise constant solution $\brho_\Delta$ to~\eqref{eq:M}
  as, for $j=1,\ldots,M$,
  \begin{equation}
    \label{eq:6}
    \rho_{j,\Delta} (t,x)  = \rho_{j,k}^n
    \quad \mbox{ for } \quad
    \left\{
      \begin{array}{l}
        t \in [t^n, t^{n+1}[,\\
        x \in [x_{k-1/2}, x_{k+1/2}[,
      \end{array}
    \right.
    \quad \mbox{ where } \quad
    \begin{array}{l}
      n= 0, \ldots, N_t -1,\\
      k \in \interi,
    \end{array}
  \end{equation}
  through a Godunov type scheme (see~\cite{AJVG2004}) together with
  operator splitting, to account for the source terms:
  \begin{lgrthm}
    \label{alg:1}
    \begin{align}
      \label{eq:fx}
      & F_j (x, u, w) =
        \left\{
        \begin{array}{l@{\quad \mbox{ if }}l}
          \min\left\{f_j \left(x, \min\{u,\theta^j (x)\}\right),
          f_j \left(x, \max\{w, \theta^j (x)\}\right) \right\}
          & x\neq 0,
          \\
          \min\left\{f_{\ell,j} \left(\min\{u, \theta^j_\ell\}\right),
          f_{r,j} \left(\max\{w, \theta^j_r \}\right) \right\}
          & x=0,
        \end{array}
            \right.
      \\
      &\texttt{for } n=0,\ldots, N_T-1 \nonumber
      \\
      & \quad  \texttt{for }  j=1,\ldots, M, \texttt{for } k \in \interi \nonumber
      \\
      \label{eq:scheme}
      & \quad \quad \quad
        \rho_{j,k}^{n+1/2} = \rh{j,k} - \lambda
        \left[
        F_j (x_{k+1/2}, \rh{j,k}, \rh{j,k+1}) - F_j (x_{k-1/2}, \rh{j,k-1}, \rh{j,k})
        \right]
      \\
      \nonumber
      & \quad  \texttt{end}
      \\
      & \quad  \texttt{for }  j=1,\ldots, M, \texttt{for } k \in \interi \nonumber
      \\ \label{eq:scheme2}
      &\quad\quad\quad
        \rho_{j,k}^{n+1} =
        \rho_{j,k}^{n+1/2} + \dt \, S_{j-1} (x_k, \rho^{n+1/2}_{j-1,k}, \rho^{n+1/2}_{j,k})
        - \dt \, S_j (x_k,\rho^{n+1/2}_{j,k}, \rho^{n+1/2}_{j+1,k})
      \\
      \nonumber
      & \quad \texttt{end}
      \\
      \nonumber
      &\texttt{end}
    \end{align}
  \end{lgrthm}

\begin{remark}\label{rem:resta0}
  Observe that, under hypotheses~\eqref{eq:idM}--\eqref{eq:idM1}, for
  all $n=0, \dots, N_T-1$ and $k\leq -1$ (corresponding to $x<0$), it
  holds $\rh{j,k}=0$ for all $j\not\in\M_\ell$. In particular, no wave
  can move backward into the segment $]-\infty,0[$ for
  $j\not\in\M_\ell$.  Similarly, for all $n=0, \dots, N_T-1$ and
  $k\geq 0$ (corresponding to $x>0$), it holds $\rh{j,k}=1$ for all
  $j\not\in\M_r$. In particular, no wave can move forward into the
  segment $]0,+\infty[$ for $j\not\in\M_r$.
\end{remark}

\subsection{Positivity and upper bound}
\label{sec:pos}

We prove that, under a suitable CFL condition, if the initial data
take values in the interval $[0,1]$, then also the approximate
solution constructed via Algorithm~\ref{alg:1} attains values in the
same interval $[0,1]$.

\begin{lemma}
  \label{lem:pos}
  Let $\brho_o \in \L\infty (\reali; [0,1]^M)$. Assume that
  \begin{equation}
    \label{eq:cfl}
    \lambda \, \mathcal{V}\leq \frac12,
  \end{equation}
  with $\mathcal{V}$ as in~\eqref{eq:normav}.  Then, for all $t>0$ and
  $x \in \reali$, the piece-wise constant approximate solution
  $\brho_\Delta$ constructed through Algorithm~\ref{alg:1} is such
  that $0 \leq \rho_{j, \Delta} (t,x) \leq 1$, for all $j=1,\dots, M$.
\end{lemma}

\begin{proof}
  By induction, assume that $0 \leq \rh{j,k} \leq 1$ for all
  $k \in \interi$ and $j=1,\dots, M$.  Consider~\eqref{eq:scheme}: it
  is well known that, for a Godunov type scheme with discontinuous
  flux function, it holds $0 \leq \rho_{j,k}^{n+1/2} \leq 1$,
  see~\cite[Lemma~4.3]{AJVG2004}. We now focus on the remaining step,
  involving the source term. In particular, fix $k\geq 0$,
  corresponding to $x>0$, the other case being entirely similar.
  Exploiting~\eqref{eq:4}, equation~\eqref{eq:scheme2} reads
  \begin{displaymath}
    \rho_{j,k}^{n+1} = \rho_{j,k}^{n+1/2}
    + \dt \, \sr{j-1} ( \rho^{n+1/2}_{j-1,k}, \rho^{n+1/2}_{j,k})
    - \dt \, \sr{j} (\rho^{n+1/2}_{j,k}, \rho^{n+1/2}_{j+1,k}).
  \end{displaymath}
  To improve readability, in what follows we omit the index
  $n+1/2$. Moreover, we take into account a \emph{complete} case, in
  which the source term contains the contributions from both the
  previous and the subsequent lane. Without loss of generality, we
  take $j=2$ and we assume both
  $\sr{1} ( \rho_{1,k}, \rho_{2,k}) \not= 0$ and
  $\sr{2} (\rho_{2,k}, \rho_{3,k}) \not= 0$.  By~\eqref{eq:scheme2}
  and~\eqref{eq:2} we obtain
  \begin{align}
    \label{eq:7}
    \rho_{2,k}^{n+1} = \
    & \rho_{2,k}
      + \dt \, \sr{1} ( \rho_{1,k}, \rho_{2,k})
      - \dt \, \sr{2} (\rho_{2,k}, \rho_{3,k})
    \\ \nonumber
    = \
    & \rho_{2,k} + \dt \left[
      \left( v_{r,2} (\rho_{2,k}) - v_{r,1} (\rho_{1,k}) \right) ^+ \rho_{1,k}
      -  \left( v_{r,2} (\rho_{2,k}) - v_{r,1} (\rho_{1,k}) \right) ^- \rho_{2,k}\right]
    \\ \nonumber
    &\qquad \!\! - \dt \left[
      \left( v_{r,3} (\rho_{3,k}) - v_{r,2} (\rho_{2,k}) \right) ^+ \rho_{2,k}
      - \left( v_{r,3} (\rho_{3,k}) - v_{r,2} (\rho_{2,k}) \right) ^- \rho_{3,k}
      \right].
  \end{align}
  There are four possibilities:
  \begin{center}
    \begin{tabular}{c|c|c}
      & $\boldsymbol{v_{r,2} (\rho_{2,k}) \geq v_{r,1} (\rho_{1,k})}$
      & $\boldsymbol{v_{r,2} (\rho_{2,k}) < v_{r,1} (\rho_{1,k})}$
      \\
      \hline
      $\boldsymbol{v_{r,3} (\rho_{3,k}) \geq v_{r,2} (\rho_{2,k}) }$
      &
        Case~\ref{item:++}
      &
        Case~\ref{item:-+}
      \\
      \hline
      $\boldsymbol{v_{r,3} (\rho_{3,k}) < v_{r,2} (\rho_{2,k}) }$
      & Case~\ref{item:+-}
      & Case~\ref{item:--}
    \end{tabular}
  \end{center}
  We analyse them in details.
  \begin{enumerate}[label=\textbf{\Alph*.}]
  \item\label{item:++} Equation~\eqref{eq:7} reads
    \begin{align*}
      \rho_{2,k}^{n+1} =  \
      & \rho_{2,k} +
        \dt  \left( v_{r,2} (\rho_{2,k}) - v_{r,1} (\rho_{1,k}) \right) \rho_{1,k}
        - \dt  \left( v_{r,3} (\rho_{3,k}) - v_{r,2} (\rho_{2,k}) \right) \rho_{2,k}
      \\
      \geq \
      & \rho_{2,k}
        - \dt  \left( v_{r,3} (\rho_{3,k}) - v_{r,2} (\rho_{2,k}) \right) \rho_{2,k}
      \\
      \geq \
      & \rho_{2,k} \left( 1 - \dt \, v_{r,3} (\rho_{3,k}) \right)
      \\
      \geq \
      & \rho_{2,k} \left( 1 - \dt \, V_{\max}\right)
      \\
      \geq \
      & 0,
    \end{align*}
    by the CFL condition~\eqref{eq:cfl}, since $\dx <1$. Moreover,
    since $v_{r,2}(1) =0$ and $\rho_{2,k}\leq 1$,
    \begin{align*}
      \rho_{2,k}^{n+1} =  \
      & \rho_{2,k} +
        \dt  \left( v_{r,2} (\rho_{2,k}) - v_{r,1} (\rho_{1,k}) \right) \rho_{1,k}
        - \dt  \left( v_{r,3} (\rho_{3,k}) - v_{r,2} (\rho_{2,k}) \right) \rho_{2,k}
      \\
      \leq \
      & \rho_{2,k} +
        \dt \,  v_{r,2} (\rho_{2,k}) \, \rho_{1,k}
        + \dt \,  v_{r,2} (\rho_{2,k}) \,  \rho_{2,k}
      \\
      = \
      &  \rho_{2,k} +
        \dt \, \vr{2}' (\sigma) \left(\rho_{2,k} - 1\right)\left(\rho_{1,k} + \rho_{2,k}\right)
      \\
      = \
      & \rho_{2,k} \left( 1 + \dt \,\vr{2}' (\sigma)  \left(\rho_{1,k} + \rho_{2,k}\right) \right)
        - \dt \,\vr{2}' (\sigma) \left(\rho_{1,k} + \rho_{2,k}\right)
      \\
      \leq \
      & 1 + \dt \,\vr{2}' (\sigma) \left(\rho_{1,k} + \rho_{2,k}\right)
        - \dt \,\vr{2}' (\sigma) \left(\rho_{1,k} + \rho_{2,k}\right)
      \\
      = \
      & 1,
    \end{align*}
    with $\sigma \in \,]\rho_{2,k},1[$ and we exploit the fact that
    $1 + \dt \,\vr{2}' \left(\rho_{1,k} + \rho_{2,k}\right) \geq 0$,
    due to the CFL condition~\eqref{eq:cfl}.

  \item\label{item:+-} By equation~\eqref{eq:7} and the hypotheses on
    the signs, it follows immediately that
    \begin{displaymath}
      \rho_{2,k}^{n+1} =
      \rho_{2,k} + \dt  \left( v_{r,2} (\rho_{2,k}) - v_{r,1} (\rho_{1,k}) \right) \rho_{1,k}
      - \dt  \left( v_{r,3} (\rho_{3,k}) - v_{r,2} (\rho_{2,k}) \right) \rho_{3,k}
      \geq 0.
    \end{displaymath}
    Moreover, since $v_{d,2} (1)=0$ and $\rho_{2,k}\leq 1$, we get
    \begin{align*}
      \rho_{2,k}^{n+1} = \
      & \rho_{2,k} + \dt  \left( v_{r,2} (\rho_{2,k}) - v_{r,1} (\rho_{1,k}) \right) \rho_{1,k}
        - \dt  \left( v_{r,3} (\rho_{3,k}) - v_{r,2} (\rho_{2,k}) \right) \rho_{3,k}
      \\
      \leq  \
      & \rho_{2,k} +
        \dt \,  v_{r,2} (\rho_{2,k}) \, \rho_{1,k}
        + \dt \,  v_{r,2} (\rho_{2,k}) \,  \rho_{3,k}
      \\
      = \
      &  \rho_{2,k} +
        \dt \, \vr{2}' (\sigma) \left(\rho_{2,k} - 1\right)\left(\rho_{1,k} + \rho_{3,k}\right)
      \\
      = \
      & \rho_{2,k} \left( 1 + \dt \,\vr{2}' (\sigma) \left(\rho_{1,k} + \rho_{3,k}\right) \right)
        - \dt \,\vr{2}' (\sigma) \left(\rho_{1,k} + \rho_{3,k}\right)
      \\
      \leq \
      & 1 + \dt \,\vr{2}' (\sigma) \left(\rho_{1,k} + \rho_{3,k}\right)
        - \dt \,\vr{2}' (\sigma) \left(\rho_{1,k} + \rho_{3,k}\right)
      \\
      = \
      & 1,
    \end{align*}
    where $\sigma \in\, ]\rho_{2,k},1[$.

  \item\label{item:-+} By equation~\eqref{eq:7} and the hypotheses on
    the sign, we get
    \begin{align*}
      \rho_{2,k}^{n+1} = \
      &  \rho_{2,k} + \dt \left[
        \left( v_{r,2} (\rho_{2,k}) - v_{r,1} (\rho_{1,k}) \right) \rho_{2,k}
        - \left( v_{r,3} (\rho_{3,k}) - v_{r,2} (\rho_{2,k}) \right) \rho_{2,k}
        \right]
      \\
      \geq \
      & \rho_{2,k} \left(1 - \dt \, v_{r,1} (\rho_{1,k})  - \dt \, v_{r,3} (\rho_{3,k}) \right)
      \\
      \geq
      & \rho_{2,k} (1 - 2 \, \dt \, V_{\max})
      \\
      \geq \
      & 0,
    \end{align*}
    by the CFL condition~\eqref{eq:cfl}, since $\dx <1$.  Moreover,
    since $v_{r,2} (\rho_{2,k}) - v_{r,1} (\rho_{1,k}) < 0$ and
    $v_{r,3} (\rho_{3,k}) - v_{r,2} (\rho_{2,k})\geq 0$, we get
    \begin{displaymath}
      \rho_{2,k}^{n+1} =
      \rho_{2,k} + \dt \left[
        \left( v_{r,2} (\rho_{2,k}) - v_{r,1} (\rho_{1,k}) \right) \rho_{2,k}
        - \left( v_{r,3} (\rho_{3,k}) - v_{r,2} (\rho_{2,k}) \right) \rho_{2,k}
      \right]
      \leq  \rho_{2,k} \leq 1.
    \end{displaymath}

  \item\label{item:--} By equation~\eqref{eq:7} and the CFL
    condition~\eqref{eq:cfl} we obtain
    \begin{align*}
      \rho_{2,k}^{n+1} = \
      &  \rho_{2,k} + \dt \left[
        \left( v_{r,2} (\rho_{2,k}) - v_{r,1} (\rho_{1,k}) \right) \rho_{2,k}
        -  \left( v_{r,3} (\rho_{3,k}) - v_{r,2} (\rho_{2,k}) \right)  \rho_{3,k}
        \right]
      \\
      \geq \
      &  \rho_{2,k} + \dt \left( v_{r,2} (\rho_{2,k}) - v_{r,1} (\rho_{1,k}) \right) \rho_{2,k}
      \\
      \geq \
      & \rho_{2,k} \left( 1 - \dt \, v_{r,1} (\rho_{1,k})\right)
      \\
      \geq \
      & \rho_{2,k} (1 - \dt \, V_{\max})
      \\
      \geq \
      & 0.
    \end{align*}
    Moreover, since $\vr{2}(1)=0$ and $\rho_{2,k}\leq 1$,
    \begin{align*}
      \rho_{2,k}^{n+1} = \
      &  \rho_{2,k} + \dt \left[
        \left( v_{r,2} (\rho_{2,k}) - v_{r,1} (\rho_{1,k}) \right) \rho_{2,k}
        -  \left( v_{r,3} (\rho_{3,k}) - v_{r,2} (\rho_{2,k}) \right)  \rho_{3,k}
        \right]
      \\
      \leq \
      & \rho_{2,k} +
        \dt \,  v_{r,2} (\rho_{2,k}) \, \rho_{2,k}
        + \dt \,  v_{r,2} (\rho_{2,k}) \,  \rho_{3,k}
      \\
      = \
      &  \rho_{2,k} +
        \dt \, \vr{2}' (\sigma) \left(\rho_{2,k} - 1\right)\left(\rho_{2,k} + \rho_{3,k}\right)
      \\
      = \
      & \rho_{2,k} \left( 1 + \dt \,\vr{2}' (\sigma) \left(\rho_{2,k} + \rho_{3,k}\right) \right)
        - \dt \,\vr{2}' (\sigma) \left(\rho_{1,k} + \rho_{3,k}\right)
      \\
      \leq \
      & 1 + \dt \,\vr{2}' (\sigma) \left(\rho_{2,k} + \rho_{3,k}\right)
        - \dt \,\vr{2}' (\sigma) \left(\rho_{2,k} + \rho_{3,k}\right)
      \\
      = \
      & 1,
    \end{align*}
    where $\sigma \in \,]\rho_{2,k},1[$.
  \end{enumerate}
  Hence, we conclude that $\rho_{j,k}^{n+1} \in[0,1]$ for all
  $j=1,\dots,M$ and $k \in \interi$.
\end{proof}

\subsection{\texorpdfstring{$\L1$}{L1}--bound}
\label{sec:l1}

The following Lemma shows that, if the initial datum $\brho_o$
satisfies $\NORM{\brho_o}< +\infty$, i.e.~it is in $\L1$ on the
\emph{active} lanes, the same holds for the corresponding
solution. Moreover, the $\L1$--norm \eqref{eq:norma1} is constant,
thus the total number of vehicles is preserved over time.
\begin{lemma}
  \label{lem:l1}
  Let $\brho_o \in (\L1 \cap \L\infty) (\reali; [0,1]^M)$.
  Let $\brho_o \in \L\infty (\reali; [0,1]^M)$, with
    $\NORM{\brho_o}< +\infty$ .  Under the CFL
  condition~\eqref{eq:cfl}, the piece-wise approximate solution
  $\brho_\Delta$ constructed through Algorithm~\ref{alg:1} is such
  that, for all $t>0$,
  \begin{equation}
    \label{eq:l1}
    \NORM{\brho_\Delta (t)}=
    \NORM{\brho_o}.
  \end{equation}
\end{lemma}

\begin{proof}
  By induction, assume that~\eqref{eq:l1} holds for $t^n = n \, \dt$.
  The Godunov type scheme~\eqref{eq:scheme} is conservative,
  see~\cite{AJVG2004}, hence
  \begin{displaymath}
    \NORM{\brho^{n+1/2}} =
    \dx \sum_{j\in\M_\ell}\sum_{k\leq -1} \modulo{\rho^{n+1/2}_{j,k}}
    +
    \dx \!\!\!\sum_{j\in\M_r} \sum _{k\geq 0}\modulo{\rho^{n+1/2}_{j,k}}
    =
    |\hspace{-1pt}\norma{\brho_o}\hspace{-1pt}|.
  \end{displaymath}
  Pass now to~\eqref{eq:scheme2}: by the positivity of $\brho_\Delta$,
  see Lemma~\ref{lem:pos}, and the assumptions on the source
  terms~\eqref{eq:sbordo}, it follows immediately that
  $ \NORM{\brho^{n+1}}= \NORM{\brho^{n+1/2}}= \NORM{\brho_o}$.
\end{proof}

\subsection{\texorpdfstring{$\L1$}{L1} continuity in time}
\label{sec:l1cont-time}

Following the idea introduced in~\cite[Lemma~3.3]{KRT2002}, we now
prove the $\L1$-continuity in time of the numerical approximation,
constructed through Algorithm~\ref{alg:1}. The result is of key
importance in the subsequent analysis.

\begin{proposition}
  \label{prop:l1cont-time}
  Let $\brho_o \in \BV (\reali; [0,1]^M)$ with
  $\NORM{\brho_o}< +\infty$ .  Assume that the CFL
  condition~\eqref{eq:cfl} holds.  Then, for $n=0, \ldots, N_T-1$
  \begin{equation}
    \label{eq:8}
    \dx \, \sum_{j=1}^M\sum_{k \in \interi} \modulo{\rho_{j,k}^{n+1} - \rho_{j,k}^n}
    \leq
    2 \, e^{4  \, \mathcal{V}\, T} \, \dt \Bigl(
    \mathcal{V} \,  \sum_{j=1}^M \tv (\rho^0_j) + M \, V_{\max}
    + 2 \, V_{\max} \, \NORM{\brho_o}
    \Bigr),
  \end{equation}
  with $V_{\max}$ and $\mathcal{V}$ as in~\eqref{eq:normav}.
\end{proposition}

\begin{remark}\label{rem:somma}
  Observe that, by Remark~\ref{rem:resta0}, the sums appearing
  in~\eqref{eq:8} are actually sums over the \emph{active} lanes only,
  the terms corresponding to fictive lanes being equal to 0. For
  example
  \begin{displaymath}
    \sum_{j=1}^M\sum_{k \in \interi} \modulo{\rho_{j,k}^{n+1} - \rho_{j,k}^n}
    =
    \sum_{j\in\mathcal{M}_\ell}\sum_{k \leq -1} \modulo{\rho_{j,k}^{n+1} - \rho_{j,k}^n}
    +
     \sum_{j\in \mathcal{M}_r}\sum_{k \geq 0} \modulo{\rho_{j,k}^{n+1} - \rho_{j,k}^n}.
   \end{displaymath}
   However, for the sake of shortness, we keep the first notation
   throughout the proof.
\end{remark}

\begin{proof}
  Fix $k\in \interi$ and $j \in \{1,\dots,M\}$. By~\eqref{eq:scheme2}
  we have:
  \begin{align}
    \nonumber
    \rho_{j,k}^{n+1} - \rh{j,k} = \
    & \rho_{j,k}^{n+1/2} - \rho_{j,k}^{n-1/2}
    \\
    \label{eq:10}
    & +\dt \, S_{j-1}(x_k, \rho_{j-1,k}^{n+1/2}, \rho_{j,k}^{n+1/2})
      - \dt \, S_{j-1}(x_k, \rho_{j-1,k}^{n-1/2}, \rho_{j,k}^{n-1/2})
    \\
    \nonumber
    & -\dt \, \, S_{j}(x_k, \rho_{j,k}^{n+1/2}, \rho_{j+1,k}^{n+1/2})
      +\dt \, \, S_{j}(x_k, \rho_{j,k}^{n-1/2}, \rho_{j+1,k}^{n-1/2}).
  \end{align}
  Observe that, by~\eqref{eq:sbordo}, terms of type
  $ \dt \left( S_{j}(x_k, \rho_{j,k}^{n+1/2}, \rho_{j+1,k}^{n+1/2}) -
    S_{j}(x_k, \rho_{j,k}^{n-1/2}, \rho_{j+1,k}^{n-1/2})\right)$ are
  non zero for $j=1,\ldots,M-1$.  For $x\in\reali$ and
  $j=1,\dots, M-1$, the function $(u,w) \mapsto S_{j}(x, u, w)$
  defined in~\eqref{eq:4}, together with~\eqref{eq:2}
  and~\eqref{eq:sghost}, is Lipschitz in both variables, with
  Lipschitz constant
  \begin{displaymath}
    K_{j} = \max\left\{
      \norma{v'_{j} (x)}_{\L\infty ([0,1])} + v_{j+1} (x, 0), \,
      \norma{v'_{j+1} (x)}_{\L\infty ([0,1])} + v_{j} (x,0)
    \right\} \leq \mathcal{V},
  \end{displaymath}
  with $\mathcal{V}$ as in~\eqref{eq:normav}. Hence, for
  $j=1,\dots,M-1$, we get
  \begin{align*}
    & \dt \,  \modulo{ S_{j}(x_k, \rho_{j,k}^{n+1/2}, \rho_{j+1,k}^{n+1/2})
      - S_{j}(x_k, \rho_{j,k}^{n-1/2}, \rho_{j+1,k}^{n-1/2})}
    \\
    \leq\
    &
      \dt \, \mathcal{V} \, \left( \modulo{\rho_{j,k}^{n+1/2} - \rho_{j,k}^{n-1/2}}
      + \modulo{\rho_{j+1,k}^{n+1/2} - \rho_{j+1,k}^{n-1/2}}
      \right).
  \end{align*}
  By~\eqref{eq:10}, taking into account also~\eqref{eq:sbordo}, we
  conclude
  \begin{align}
    \nonumber
    & \sum_{j=1}^M \modulo{\rho_{j,k}^{n+1} - \rh{j,k}}
    \\ \nonumber
    \leq \
    & \sum_{j=1}^M  \modulo{\rho_{j,k}^{n+1/2} - \rho^{n-1/2}_{j,k}}
    \\ \nonumber
    & + 2  \, \mathcal{V} \, \dt \left(
      \modulo{\rho_{1,k}^{n+1/2} - \rho^{n-1/2}_{1,k}}
      + 2 \sum_{j=2}^{M-1} \modulo{\rho_{j,k}^{n+1/2} - \rho^{n-1/2}_{j,k}}
      +  \modulo{\rho_{M,k}^{n+1/2} - \rho^{n-1/2}_{M,k}}
      \right)
    \\ \nonumber
    \leq \
    & (1+ 4  \, \mathcal{V}\, \dt)
      \sum_{j=1}^M  \modulo{\rho_{j,k}^{n+1/2} - \rho^{n-1/2}_{j,k}}
    \\  \label{eq:9}
    \leq \
    & e^{4  \, \mathcal{V}\, \dt}
      \sum_{j=1}^M  \modulo{\rho_{j,k}^{n+1/2} - \rho^{n-1/2}_{j,k}}.
  \end{align}
  Exploit now~\eqref{eq:scheme}: we have, for fixed
  $j \in \{1,\dots,M\}$ and $k \in \interi$,
  \begin{align}
    \nonumber
    \rho_{j,k}^{n+1/2} - \rho^{n-1/2}_{j,k} = \
    & \rh{j,k} - \rho_{j,k}^{n-1}
      - \lambda
      \left[
      F_j (x_{k+1/2}, \rh{j,k}, \rh{j,k+1}) - F_j (x_{k-1/2}, \rh{j,k-1}, \rh{j,k})
      \right]
    \\
    \label{eq:12}
    & + \lambda
      \left[
      F_j (x_{k+1/2}, \rho^{n-1}_{j,k}, \rho^{n-1}_{j,k+1})
      - F_j (x_{k-1/2}, \rho^{n-1}_{j,k-1}, \rho^{n-1}_{j,k})
      \right].
  \end{align}
  We closely follow the proof
  of~\cite[Lemma~3.3]{KRT2002}. In~\eqref{eq:12} add and subtract
  $\lambda \, F_j (x_{k+1/2}, \rh{j,k}, \rho^{n-1}_{j,k+1})$ and
  $ \lambda \, F_j (x_{k-1/2}, \rh{j,k-1}, \rho^{n-1}_{j,k})$ and,
  setting
  \begin{align}
    \label{eq:alfa}
    \alpha_{j,k}^n = \
    & \left\{
      \begin{array}{l@{\quad \mbox{ if }}l}
        - \lambda \,
        \dfrac{F_j (x_{k-1/2}, \rh{j,k-1}, \rh{j,k})
        - F_j (x_{k-1/2}, \rh{j,k-1}, \rho^{n-1}_{j,k})}{\rh{j,k} - \rho^{n-1}_{j,k}}
        & \rh{j,k} \neq \rho^{n-1}_{j,k},
        \\
        0 & \rh{j,k} = \rho^{n-1}_{j,k},
      \end{array}
            \right.
    \\
    \label{eq:beta}
    \beta_{j,k}^n =  \
    & \left\{
      \begin{array}{l@{\qquad \mbox{ if }}l}
        \lambda \,
        \dfrac{F_j (x_{k+1/2}, \rh{j,k}, \rho^{n-1}_{j,k+1})
        - F_j (x_{k+1/2}, \rho^{n-1}_{j,k}, \rho^{n-1}_{j,k+1})}{\rh{j,k} - \rho^{n-1}_{j,k}}
        & \rh{j,k} \neq \rho^{n-1}_{j,k},
        \\
        0 & \rh{j,k} = \rho^{n-1}_{j,k},
      \end{array}
            \right.
  \end{align}
  rearrange the resulting expression to obtain
  \begin{equation}
    \label{eq:11}
    \begin{aligned}
      \rho_{j,k}^{n+1/2} - \rho^{n-1/2}_{j,k} = \ & \left(\rh{j,k} -
        \rho^{n-1}_{j,k}\right) \left( 1 - \alpha_{j,k}^n -
        \beta_{j,k}^n\right)
      \\
      & + \alpha_{j,k+1}^n \left(\rh{j,k+1} -
        \rho^{n-1}_{j,k+1}\right) + \beta_{j,k-1}^n \left(\rh{j,k-1} -
        \rho^{n-1}_{j,k-1}\right).
    \end{aligned}
  \end{equation}
  Since the numerical flux $F_j$ defined in~\eqref{eq:fx} is non
  decreasing in the second variable and non increasing in the third,
  we get $\alpha_{j,k}^n, \, \beta_{j,k}^n \geq 0$ for all
  $j=1,\dots,M$ and $k \in \interi$. Moreover, $F_j(x, \cdot, \cdot)$
  is Lipschitz in both arguments, for $x\in \reali$, with Lipschitz
  constant bounded by $\mathcal{V}$ as
  in~\eqref{eq:normav}. Therefore,
  \begin{align*}
    \beta_{j,k}^n = \
    & \frac{\lambda}{\rh{j,k} - \rho^{n-1}_{j,k}} \left(
      F_j (x_{k+1/2}, \rh{j,k}, \rho^{n-1}_{j,k+1})
      - F_j (x_{k+1/2}, \rho^{n-1}_{j,k}, \rho^{n-1}_{j,k+1})
      \right)
    \\
    \leq \
    & \frac{\lambda}{\rh{j,k} - \rho^{n-1}_{j,k}} \,
      \mathcal{V} \, \modulo{\rh{j,k} - \rho^{n-1}_{j,k}}
      = \lambda \, \mathcal{V} \leq \frac12,
  \end{align*}
  by the CFL condition~\eqref{eq:cfl}. A similar argument applies to
  $\alpha^n_{j,k}$. As a consequence,
  $1 - \alpha^n_{j,k} - \beta^n_{j,k} \geq 0$, thus all the
  coefficients appearing in~\eqref{eq:11} are positive and so
  \begin{align}
    \nonumber
    \sum_{k\in \interi} \modulo{\rho_{j,k}^{n+1/2} - \rho^{n-1/2}_{j,k}} \leq \
    & \sum_{k\in \interi} \modulo{\rh{j,k} - \rho^{n-1}_{j,k}}
      \left( 1 - \alpha_{j,k}^n - \beta_{j,k}^n\right)
    \\ \nonumber
    & + \sum_{k\in \interi} \alpha_{j,k+1}^n \, \modulo{\rh{j,k+1} - \rho^{n-1}_{j,k+1}}
      + \sum_{k\in \interi} \beta_{j,k-1}^n \,  \modulo{\rh{j,k-1} - \rho^{n-1}_{j,k-1}}
    \\ \label{eq:13}
    = \
    & \sum_{k\in \interi} \modulo{\rh{j,k} - \rho^{n-1}_{j,k}}.
  \end{align}
  Collecting together~\eqref{eq:9} and~\eqref{eq:13} leads to
  \begin{displaymath}
    \sum_{j=1}^M  \sum_{k\in \interi}  \modulo{\rho_{j,k}^{n+1} - \rh{j,k}}
    \leq
    e^{4 \, \mathcal{V}\, \dt }  \sum_{j=1}^M  \sum_{k\in \interi}
    \modulo{\rho_{j,k}^{n+1/2} - \rho^{n-1/2}_{j,k}}
    \leq
    e^{4  \, \mathcal{V}\, \dt}  \sum_{j=1}^M  \sum_{k\in \interi}
    \modulo{\rh{j,k} - \rho^{n-1}_{j,k}},
  \end{displaymath}
  which applied recursively yields
  \begin{equation}
    \label{eq:14}
    \dx \sum_{j=1}^M  \sum_{k\in \interi}  \modulo{\rho_{j,k}^{n+1} - \rh{j,k}}
    \leq
    e^{4\, \mathcal{V} \, T }
    \dx \sum_{j=1}^M  \sum_{k\in \interi}  \modulo{\rho_{j,k}^{1} - \rho^0_{j,k}},
  \end{equation}
  where we also multiplied both sides of the inequality by $\dx$.

  Using~\eqref{eq:scheme} and~\eqref{eq:scheme2}, compute
  \begin{align}
    \nonumber
    \rho^{1}_{j,k} - \rho^0_{j,k} =  \
    & \rho^{1/2}_{j,k} - \rho^0_{j,k}
      + \dt \, S_{j-1} (x_{k}, \rho^{1/2}_{j-1,k}, \rho^{1/2}_{j,k})
      - \dt \, S_j (x_k, \rho^{1/2}_{j,k}, \rho^{1/2}_{j+1,k})
    \\ \label{eq:A}
    = \
    & -\lambda \, \left[
      F_j (x_{k+1/2}, \rho^0_{j,k}, \rho^0_{j,k+1})
      - F_j (x_{k-1/2}, \rho^0_{j,k-1}, \rho^0_{j,k})
      \right]
    \\ \label{eq:B}
    & + \dt \, S_{j-1} (x_{k}, \rho^{1/2}_{j-1,k}, \rho^{1/2}_{j,k})
      - \dt \, S_j (x_k, \rho^{1/2}_{j,k}, \rho^{1/2}_{j+1,k}).
  \end{align}
  Focus first on~\eqref{eq:B}: by the definition of
  $S_j$~\eqref{eq:2}--\eqref{eq:4}--\eqref{eq:sbordo}, for
  $j=1,\dots,M-1$ we have
  \begin{equation}
    \label{eq:boundS}
    \modulo{S_j (x_k, \rho^{1/2}_{j,k}, \rho^{1/2}_{j+1,k})}
    \leq
    V_{\max} \left( \rho^{1/2}_{j,k} +  \rho^{1/2}_{j+1,k}\right).
  \end{equation}
  Therefore, recalling Remark~\ref{rem:somma}, with slight abuse of notation
  \begin{align}
    \nonumber
    \dx \sum_{j=1}^M\sum_{k \in \interi} \dt
    \modulo{S_{j-1} (x_{k}, \rho^{1/2}_{j-1,k}, \rho^{1/2}_{j,k})
    - S_j (x_k, \rho^{1/2}_{j,k}, \rho^{1/2}_{j+1,k})}
    \leq \
    & \dx \, \dt \, V_{\max} \sum_{j=1}^M\sum_{k \in \interi}
      4 \, \rho^{1/2}_{j,k}
    \\ \nonumber
    = \
    & 4 \, \dt \, V_{\max} \, \NORM{\rho^{1/2}} 
    \\
    \label{eq:15}
    = \
    & 4 \, \dt \, V_{\max} \, \NORM{\brho_o},
  \end{align}
  where we use Lemma~\ref{lem:l1}.

  Pass now to~\eqref{eq:A}. Since we are interested in the sum over
  $k \in \interi$, we distinguish among four cases: $k<-1$, $k>0$,
  $k=-1$ and $k=0$.

  The first case, $k<-1$, amounts to $x_{k-1/2}<x_{k+1/2}<0$. Thus, by
  the definition of $F_j$~\eqref{eq:fx}, together with~\eqref{eq:3},
  the numerical flux does not depend on the variable $x$, namely
  \begin{displaymath}
    \mbox{for } x<0: \quad
    F_j (x, u, w) =
    \min\bigl\{f_{\ell,j} \bigl(\min\{u, \theta_\ell^j\}\bigr),
    f_{\ell,j} \left(\max\{w, \theta_\ell^j\}\right) \bigr\},
  \end{displaymath}
  and the function above is clearly Lipschitz in both $u$ and $w$,
  with Lipschitz constant $\mathcal{V}$ as in~\eqref{eq:normav},
  leading to
  \begin{equation}
    \label{eq:xneg}
    \begin{aligned}
      & \sum_{k < -1} \modulo{ F_j (x_{k+1/2}, \rho^0_{j,k},
        \rho^0_{j,k+1}) - F_j (x_{k-1/2}, \rho^0_{j,k-1},
        \rho^0_{j,k})}
      \\
      \leq \ & \mathcal{V} \sum_{k<-1} \left(\modulo{\rho^0_{j,k} -
          \rho^0_{j,k-1}} + \modulo{\rho^0_{j,k+1} - \rho^0_{j,k}}
      \right).
    \end{aligned}
  \end{equation}
  The case $k>0$ can be treated analogously, leading to
  \begin{equation}
    \label{eq:xpos}
    \begin{aligned}
      & \sum_{k >0} \modulo{ F_j (x_{k+1/2}, \rho^0_{j,k},
        \rho^0_{j,k+1}) - F_j (x_{k-1/2}, \rho^0_{j,k-1},
        \rho^0_{j,k})}
      \\
      \leq \ & \mathcal{V} \, \sum_{k >0} \left(\modulo{\rho^0_{j,k} -
          \rho^0_{j,k-1}} + \modulo{\rho^0_{j,k+1} - \rho^0_{j,k}}
      \right).
    \end{aligned}
  \end{equation}
  Pass now to $k=0$. Recall that $x_{-1/2}=0$. By the definition of
  $F_j$~\eqref{eq:fx}, together with~\eqref{eq:3}, we have
  \begin{align*}
    F_j (x_{1/2}, \rho^0_{j,0},  \rho^0_{j,1})
    - F_j (x_{-1/2}, \rho^0_{j,-1}, \rho^0_{j,0})
    = \
    & \min\bigl\{f_{r,j} \left(\min\{\rho^0_{j,0}, \theta_r^j\}\right),
      f_{r,j} \left(\max\{\rho^0_{j,1}, \theta_r^j\}\right) \bigr\}
    \\
    & - \min\bigl\{f_{\ell,j} \left(\min\{\rho^0_{j,-1}, \theta_\ell^j\}\right),
      f_{r,j} \left(\max\{\rho^0_{j,0},  \theta_r^j\}\right) \bigr\}.
  \end{align*}
  We immediately get
  \begin{equation}
    \label{eq:x0}
    \modulo{F_j (x_{1/2}, \rho^0_{j,0},  \rho^0_{j,1})\!
      - \! F_j (x_{-1/2}, \rho^0_{j,-1}, \rho^0_{j,0})} \leq
    \norma{f_j}_{\L\infty} \leq V_{\max},
  \end{equation}
  with $V_{\max}$ as in~\eqref{eq:normav}. The case $k=-1$ follows
  analogously.

  Hence, collecting together~\eqref{eq:xneg}, \eqref{eq:xpos}
  and~\eqref{eq:x0} and using the fact that $\lambda \, \dx = \dt$, we
  obtain
  \begin{align}
    \nonumber
    &  \dx\sum_{k \in \interi}  \lambda \,
      \modulo{ F_j (x_{k+1/2}, \rho^0_{j,k},  \rho^0_{j,k+1})
      - F_j (x_{k-1/2}, \rho^0_{j,k-1},  \rho^0_{j,k})}
    \\
    \nonumber
    \leq \
    & \mathcal{V} \, \dt  \sum_{k \in \interi} \left(\modulo{\rho^0_{j,k} - \rho^0_{j,k-1}}
      + \modulo{\rho^0_{j,k+1} - \rho^0_{j,k}}
      \right) + 2 \, \dt \, V_{\max}
    \\
    \label{eq:16}
    \leq \
    & 2 \, \mathcal{V} \, \dt  \sum_{k \in \interi}\modulo{\rho^0_{j,k} - \rho^0_{j,k-1}}
      +2 \, \dt \, V_{\max}.
  \end{align}
  By~\eqref{eq:A}--\eqref{eq:B}, insert~\eqref{eq:15}
  and~\eqref{eq:16} into~\eqref{eq:14}:
  \begin{align*}
    \dx \sum_{j=1}^M  \sum_{k \in \interi}
    \modulo{\rho^{n+1}_{j,k} - \rho^n_{j,k}}
    \leq \
    & 2 \, e^{4 \, \mathcal{V} \, T} \, \dt \Bigl(
      \mathcal{V} \,  \sum_{j=1}^M \tv (\rho^0_j) + M \, V_{\max}
      + 2 \, V_{\max} \, \NORM{\brho_o}
      \Bigr),
  \end{align*}
  concluding the proof.
\end{proof}

\subsection{Spatial \texorpdfstring{$\BV$}{BV} bound }
\label{sec:bv}

We follow the idea of~\cite[Lemma~4.2]{BGKT2008} of providing a
\emph{local} spatial $\BV$ bound, in the sense that the estimate
in~\eqref{eq:bv} below blows up if one of the endpoints of the
interval $[a,b]$ approaches $x=0$.

\begin{lemma}
  \label{lem:bv}
  Let $\brho_o \in \BV (\reali; [0,1]^M)$ with
  $\NORM{\brho_o}< +\infty$. Assume that the CFL
  condition~\eqref{eq:cfl} holds. For any interval
  $[a,b] \subseteq \reali$ such that $0 \notin [a,b] $, fix $s>0$ such
  that $2 \, s < \min\{\modulo{a}, \modulo{b}\}$ and $s>\dx$. Then,
  for any $n=1, \ldots, N_T-1$
  the following estimate holds:
  \begin{equation}
    \label{eq:bv}
    \sum_{j=1}^M \sum_{k \in \mathbf{K}_a^b} \modulo{\rh{j,k+1} - \rh{j,k}}
    \leq e^{4  \, \mathcal{V} \, T} \biggl( \sum_{j=1}^M \tv (\rho_{o,j})
    + 8\, M \, V_{\max}\, T + \frac{2 \, C}{s}\biggr),
  \end{equation}
  with
  $\mathbf{K}_a^b = \left\{ k \in \interi \colon a \leq x_{k} \leq
    b\right\}$, $V_{\max}$ and $\mathcal{V}$ as in~\eqref{eq:normav}
  and $C$ independent of $\dx$ and $\dt$.
\end{lemma}

\begin{proof}
  Let
  \begin{align*}
    \mathcal{A}_\Delta = \
    & \left\{k \in \interi \colon x_{k-1/2} \in [a-s-\dx, a]\right\},
    &
      \mathcal{B}_\Delta = \
    & \left\{k \in \interi \colon x_{k+1/2} \in [b, b+s+\dx]\right\}.
  \end{align*}
  By the assumptions on $s$, observe that there are at least 2
  elements in each of the sets above,
  i.e.~$\modulo{\mathcal{A}_\Delta}, \, \modulo{\mathcal{B}_\Delta}
  \geq 2$. Moreover, $\modulo{\mathcal{A}_\Delta} \, \dx \geq s$ and
  $\modulo{\mathcal{B}_\Delta} \dx \geq s$. Furthermore, notice that
  \begin{itemize}
  \item if $0<a<b$: it holds $x_{k-1/2}>0$ for any
    $k\in\mathcal{A}_\Delta$;

  \item if $a<b<0$: it holds $x_{k+1/2}<0$ for any
    $k \in \mathcal{B}_\Delta$.
  \end{itemize}
  By Proposition~\ref{prop:l1cont-time}, there exists a constant $C$
  such that
  \begin{displaymath}
    \dx \sum_{n=0}^{N_T -1} \sum_{j=1}^M \, \sum_{k \in \interi}
    \modulo{\rho^{n+1}_{j,k} - \rh{j,k}} \leq C,
  \end{displaymath}
  with
  $C = 2 \, T \, e^{4 \, \mathcal{V} \, T} \left( \mathcal{V} \, \tv
    (\brho_o) + M \, V_{\max} + 2 V_{\max} \,
    \NORM{\brho_o}\right)$. Hence, when restricting the sum over $k$
  in the set $\mathcal{A}_\Delta$, respectively $\mathcal{B}_\Delta$,
  it clearly follows that
  \begin{align}
    \label{eq:19}
    \dx \sum_{n=0}^{N_T -1} \sum_{j=1}^M \, \sum_{k \in \mathcal{A}_\Delta}
    \modulo{\rho^{n+1}_{j,k} - \rh{j,k}} \leq \
    & C,
    &
      \dx \sum_{n=0}^{N_T -1} \sum_{j=1}^M \, \sum_{k \in \mathcal{B}_\Delta}
      \modulo{\rho^{n+1}_{j,k} - \rh{j,k}} \leq \
    & C.
  \end{align}
  Choose $k_a \in \mathcal{A}_\Delta$ and $k_b$ with
  $k_b+1 \in \mathcal{B}_\Delta$ such that
  \begin{align*}
    \sum_{n=0}^{N_T -1} \sum_{j=1}^M
    \modulo{\rho^{n+1}_{j,k_a} - \rh{j,k_a}}
    = \
    & \min_{k \in \mathcal{A}_\Delta}
      \sum_{n=0}^{N_T -1} \sum_{j=1}^M
      \modulo{\rho^{n+1}_{j,k} - \rh{j,k}},
    \\
    \sum_{n=0}^{N_T -1} \sum_{j=1}^M
    \modulo{\rho^{n+1}_{j,k_b+1} - \rh{j,k_b+1}}
    = \
    & \min_{k \in \mathcal{B}_\Delta}
      \sum_{n=0}^{N_T -1} \sum_{j=1}^M
      \modulo{\rho^{n+1}_{j,k} - \rh{j,k}},.
  \end{align*}
  Thus, by~\eqref{eq:19},
  \begin{equation}
    \label{eq:20}
    \begin{aligned}
      \sum_{n=0}^{N_T -1} \sum_{j=1}^M \modulo{\rho^{n+1}_{j,k_a} -
        \rh{j,k_a}} \leq \ & \frac{C}{\modulo{\mathcal{A}_\Delta} \,
        \dx} \leq \frac{C}{s},
      \\
      \sum_{n=0}^{N_T -1} \sum_{j=1}^M \modulo{\rho^{n+1}_{j,k_b+1} -
        \rh{j,k_b+1}} \leq \ & \frac{C}{\modulo{\mathcal{B}_\Delta} \,
        \dx} \leq \frac{C}{s}.
    \end{aligned}
  \end{equation}
  In view of the next steps, observe that
  \begin{equation}
    \label{eq:decomp}
    \sum_{k=k_a}^{k_b}\modulo{\rho^{n+1}_{j,k+1} - \rho^{n+1}_{j,k}}
    =
    \modulo{\rho^{n+1}_{j,k_a+1} - \rho^{n+1}_{j,k_a}}
    + \sum_{k=k_a+1}^{k_b-1}\modulo{\rho^{n+1}_{j,k+1} - \rho^{n+1}_{j,k}}
    + \modulo{\rho^{n+1}_{j,k_b+1} - \rho^{n+1}_{j,k_b}}.
  \end{equation}
  Focus on the central sum on the right hand side
  of~\eqref{eq:decomp}. By~\eqref{eq:scheme2}, for $k_a < k < k_b$ and
  $j=1,\dots,M$, we have
  \begin{align*}
    \rho^{n+1}_{j,k+1} - \rho^{n+1}_{j,k}
    = \
    &  \rho^{n+1/2}_{j,k+1} - \rho^{n+1/2}_{j,k}
    \\
    & + \dt \left(
      S_{j-1} (x_k,  \rho^{n+1}_{j-1,k+1} - \rho^{n+1}_{j,k+1})
      - S_{j-1} (x_k,  \rho^{n+1}_{j-1,k} - \rho^{n+1}_{j,k})\right.
    \\
    & \qquad\quad \left.
      - S_{j} (x_k,  \rho^{n+1}_{j,k+1} - \rho^{n+1}_{j+1,k+1})
      + S_{j} (x_k,  \rho^{n+1}_{j,k} - \rho^{n+1}_{j+1,k})
      \right).
  \end{align*}
  By the Lipschitz continuity of the map $(u,w) \mapsto S_j (x,u,w)$
  for $x\in \reali$ and $j=1,\dots, M-1$, we get
  \begin{align}
    \label{eq:17}
    \sum_{j=1}^M\modulo{ \rho^{n+1}_{j,k+1} - \rho^{n+1}_{j,k}}
    \leq \
    & (1 + 4 \, \mathcal{V} \,  \dt ) \sum_{j=1}^M \modulo{ \rho^{n+1/2}_{j,k+1} - \rho^{n+1/2}_{j,k}}
      .
  \end{align}
  Fix now $j\in \{1,\dots,M\}$. Recall that for all
  $k_a \leq k \leq k_b$ either $x_{k-1/2}>0$ or
  $x_{k+1/2}<0$. Therefore, when applying~\eqref{eq:scheme}, observe
  that the numerical flux $F_j$~\eqref{eq:fx} is never computed at
  $x=0$, leading to
  \begin{equation}
    \label{eq:schemeG}
    \rho^{n+1/2}_{j,k} = \rh{j,k} - \lambda \left[
      G_{d,j} (\rh{j,k}, \rh{j,k+1}) - G_{d,j} (\rh{j,k-1}, \rh{j,k})
    \right],
  \end{equation}
  for $d=\ell, r$, with
  \begin{equation}
    \label{eq:Gdj}
    G_{d,j} (u,w) = \min\bigl\{
    f_{d,j} \left(\min\{u, \theta_d^j\}\right), \, f_{d,j}\left(\max\{w, \theta_d^j\}\right)
    \bigr\}.
  \end{equation}
  Clearly, it is $d=\ell$ whenever $a<b<0$ and $d=r$ whenever $0<a<b$.
  Adding and subtracting
  $\lambda \, G_{d,j} (\rh{j,k}, \rh{j,k}) = \lambda \, f_{d,j}
  (\rh{j,k})$ into \eqref{eq:schemeG} and setting
  \begin{align}
    \label{eq:gamma}
    \gamma_{d,j,k}^n = \
    & \left\{
      \begin{array}{l@{\quad \mbox{ if }}l}
        -\lambda \,
        \dfrac{G_{d,j} (\rh{j,k+1},\rh{j,k}) - G_{d,j} (\rh{j,k}, \rh{j,k})}{\rh{j,k+1}- \rh{j,k}}
        & \rh{j,k+1} \neq \rh{j,k},
        \\
        0
        & \rh{j,k+1} = \rh{j,k},
      \end{array}
          \right.
    \\
    \label{eq:delta}
    \delta_{d,j,k}^n = \
    & \left\{
      \begin{array}{l@{\qquad \mbox{ if }}l}
        \lambda \,
        \dfrac{G_{d,j} (\rh{j,k},\rh{j,k}) - G_{d,j} (\rh{j,k-1}, \rh{j,k})}{\rh{j,k}- \rh{j,k-1}}
        & \rh{j,k} \neq \rh{j,k-1},
        \\
        0
        & \rh{j,k} = \rh{j,k-1},
      \end{array}
          \right.
  \end{align}
  we can rearrange~\eqref{eq:schemeG} to get
  \begin{equation}
    \label{eq:scomp}
    \rho^{n+1/2}_{j,k} = \rh{j,k} + \gamma_{d,j,k}^n \left(\rh{j,k+1} - \rh{j,k}\right)
    - \delta_{d,j,k}^n \left(\rh{j,k} - \rh{j,k-1}\right).
  \end{equation}
  The function $G_{d,j}$ is non decreasing in the first argument and
  non increasing in the second, so that we easily get
  $\gamma_{d,j,k}^n , \, \delta_{d,j,k}^n \geq 0$. Furthermore,
  $G_{d,j}$ is Lipschitz continuous in both variables, with the same
  Lipschitz constant $\mathcal{V}$~\eqref{eq:normav} as $F_j$: by the
  CFL condition~\eqref{eq:cfl}
  \begin{align*}
    \gamma_{d,j,k}^n\leq \
    & \lambda \, \mathcal{V} \leq \frac12,
    &
      \delta_{d,j,k}^n\leq \
    & \lambda \, \mathcal{V} \leq \frac12,
  \end{align*}
  and hence $\gamma_{d,j,k}^n + \delta_{d,j,k+1}^n\leq 1$.  Therefore,
  for $k_a< k < k_b$
  \begin{equation}
    \label{eq:21}
    \begin{aligned}
      \rho^{n+1/2}_{j,k+1} - \rho^{n+1/2}_{j,k} = \ & \left(\rh{j,k+1}
        - \rh{j,k}\right) \left(1 - \gamma_{d,j,k}^n -
        \delta_{d,j,k+1}^n\right)
      \\
      & + \gamma_{d,j, k+1}^n \left(\rh{j,k+2} - \rh{j,k+1}\right) +
      \delta_{d,j,k}^n \left(\rh{j,k} - \rh{j,k-1}\right).
    \end{aligned}
  \end{equation}
  We are left with the boundary terms in~\eqref{eq:decomp}. Fix
  $j \in \{1,\dots, M\}$. For $k=k_a$, applying
  first~\eqref{eq:scheme2} then~\eqref{eq:scheme}, in the form
  of~\eqref{eq:scomp}, we have
  \begin{align*}
    & \rho^{n+1}_{j,k_a+1} - \rho^{n+1}_{j,k_a}
    \\
    = \
    &  \rho^{n+1/2}_{j,k_a+1}
      + \dt \, S_{j-1} (x_{k_a+1}, \rho^{n+1/2}_{j-1,k_a+1}, \rho^{n+1/2}_{j,k_a+1})
      - \dt \, S_{j} (x_{k_a+1}, \rho^{n+1/2}_{j,k_a+1}, \rho^{n+1/2}_{j+1,k_a+1})
      -  \rho^{n+1}_{j,k_a}
    \\
    = \
    & \rh{j,k_a+1} + \gamma^n_{d,j,k_a+1} \left(\rh{j,k_a+2} - \rh{j,k_a+1}\right)
      - \delta^n_{d,j,k_a+1}\left(\rh{j,k_a+1} - \rh{j,k_a}\right)
    \\
    & + \dt \, S_{j-1} (x_{k_a+1}, \rho^{n+1/2}_{j-1,k_a+1}, \rho^{n+1/2}_{j,k_a+1})
      - \dt \, S_{j} (x_{k_a+1}, \rho^{n+1/2}_{j,k_a+1}, \rho^{n+1/2}_{j+1,k_a+1})
      -  \rho^{n+1}_{j,k_a}.
  \end{align*}
  Add and subtract $\rh{j,k_a}$, then take the absolute value and sum
  over $j=1,\dots,M$: exploiting~\eqref{eq:boundS} leads to
  \begin{equation}
    \label{eq:ka}
    \begin{aligned}
      \sum_{j=1}^M \modulo{ \rho^{n+1}_{j,k_a+1} - \rho^{n+1}_{j,k_a}
      } \leq \ & \sum_{j=1}^M \modulo{\rho^{n+1}_{j,k_a} -\ \rh{j,k_a}
      } + \sum_{j=1}^M (1-\delta^n_{d,j,k_a+1}) \modulo{\rh{j,k_a+1}
        -\rh{j,k_a}}
      \\
      & + \sum_{j=1}^M \gamma^n_{d,j,k_a+1} \modulo{\rh{j,k_a+2} -
        \rh{j,k_a+1}} + 4 \, \dt \, V_{\max} \sum_{j=1}^M
      \rho^{n+1/2}_{j,k_a+1}.
    \end{aligned}
  \end{equation}
  Proceed similarly for $k=k_b$:
  \begin{align*}
    & \rho^{n+1}_{j,k_b+1} - \rho^{n+1}_{j,k_b}
    \\
    = \
    &  \rho^{n+1}_{j,k_b+1} - \rho^{n+1/2}_{j,k_b}
      - \dt \, S_{j-1} (x_{k_b}, \rho^{n+1/2}_{j-1,k_b}, \rho^{n+1/2}_{j,k_b})
      + \dt \, S_{j} (x_{k_b}, \rho^{n+1/2}_{j,k_b}, \rho^{n+1/2}_{j+1,k_b})
    \\
    = \
    & \rho^{n+1}_{j,k_b+1} - \rh{j,k_b}
      - \gamma^n_{d,j,k_b} \left(\rh{j,k_b+1} - \rh{j,k_b}\right)
      + \delta^n_{d,j,k_b}\left(\rh{j,k_b} - \rh{j,k_b-1}\right)
    \\
    &   - \dt \, S_{j-1} (x_{k_b}, \rho^{n+1/2}_{j-1,k_b}, \rho^{n+1/2}_{j,k_b})
      + \dt \, S_{j} (x_{k_b}, \rho^{n+1/2}_{j,k_b}, \rho^{n+1/2}_{j+1,k_b}).
  \end{align*}
  Now add and subtract $\rh{j,k_b+1}$, take the absolute value and sum
  over $j=1,\dots,M$:
  \begin{equation}
    \label{eq:kb}
    \begin{aligned}
      \sum_{j=1}^M \modulo{ \rho^{n+1}_{j,k_b+1} - \rho^{n+1}_{j,k_b}
      } \leq \ & \sum_{j=1}^M \modulo{\rho^{n+1}_{j,k_b+1} -\
        \rh{j,k_b+1} } + \sum_{j=1}^M (1-\gamma^n_{d,j,k_b})
      \modulo{\rh{j,k_b+1} -\rh{j,k_b}}
      \\
      & + \sum_{j=1}^M \delta^n_{d,j,k_b} \modulo{\rh{j,k_b+1} -
        \rh{j,k_b}} + 4 \, \dt \, V_{\max} \sum_{j=1}^M
      \rho^{n+1/2}_{j,k_b}.
    \end{aligned}
  \end{equation}

  By~\eqref{eq:decomp}, collect together~\eqref{eq:17}, \eqref{eq:21},
  \eqref{eq:ka} and~\eqref{eq:kb}: since all the coefficients
  appearing there are positive, we obtain
  \begin{align*}
    &  \sum_{j=1}^M \, \sum_{k=k_a}^{k_b}
      \modulo{\rho^{n+1}_{j,k+1} - \rho^{n+1}_{j,k}}
    \\
    \leq \
    & \sum_{j=1}^M \modulo{\rho^{n+1}_{j,k_a} -\ \rh{j,k_a} }
      + \sum_{j=1}^M (1-\delta^n_{d,j,k_a+1}) \modulo{\rh{j,k_a+1} -\rh{j,k_a}}
      + \sum_{j=1}^M \gamma^n_{d,j,k_a+1} \modulo{\rh{j,k_a+2} - \rh{j,k_a+1}}
    \\
    & + 4 \, \dt \, V_{\max} \sum_{j=1}^M \rho^{n+1/2}_{j,k_a+1}
      + e^{4\, \mathcal{V}  \, \dt } \sum_{j=1}^M \, \sum_{k=k_a+1}^{k_b-1}
      \left( 1 - \gamma_{d,j,k}^n -  \delta_{d,j,k+1}^n  \right)
      \modulo{\rh{j,k+1} - \rh{j,k}}
    \\
    & + e^{4 \, \mathcal{V}\, \dt } \sum_{j=1}^M \, \sum_{k=k_a+1}^{k_b-1}
      \gamma_{d,j,k+1}^n  \modulo{\rh{j,k+2} - \rh{j,k+1}}
      + e^{4 \, \mathcal{V} \, \dt} \sum_{j=1}^M \, \sum_{k=k_a+1}^{k_b-1}
      \delta_{d,j,k}^n   \modulo{\rh{j,k} - \rh{j,k-1}}
    \\
    & +  \sum_{j=1}^M \modulo{\rho^{n+1}_{j,k_b+1} -\ \rh{j,k_b+1} }
      + \sum_{j=1}^M (1-\gamma^n_{d,j,k_b}) \modulo{\rh{j,k_b+1} -\rh{j,k_b}}
      + \sum_{j=1}^M \delta^n_{d,j,k_b} \modulo{\rh{j,k_b+1} - \rh{j,k_b}}
    \\
    & + 4 \, \dt \, V_{\max} \sum_{j=1}^M \rho^{n+1/2}_{j,k_b}
    \\
    \leq \
    &  \sum_{j=1}^M \left( e^{4 \mathcal{V} \, \dt} \sum_{k=k_a}^{k_b}
      \modulo{\rh{j,k+1} - \rh{j,k}}
      +  \modulo{\rho^{n+1}_{j,k_a} -\ \rh{j,k_a} }
      +  \modulo{\rho^{n+1}_{j,k_b+1} -\ \rh{j,k_b+1}}\right)
      +8 \, M \, V_{max} \, \dt,
  \end{align*}
  where we exploit also Lemma~\ref{lem:pos}. Proceeding recursively we
  finally get, for $1\leq n < N_T-1$,
  \begin{align*}
    \sum_{j=1}^M \, \sum_{k=k_a}^{k_b}
    \modulo{\rho^{n+1}_{j,k+1} - \rho^{n+1}_{j,k}}
    \leq \
    &  e^{4 \, \mathcal{V} \, (n+1) \, \dt } \sum_{j=1}^M \tv (\rho_{o,j})
      + e^{4\, \mathcal{V} \, n \, \dt} \, 8 \, M  \, V_{\max} (n+1)\, \dt
    \\
    & + e^{4\, \mathcal{V} \, n \, \dt }\sum_{m=0}^n\sum_{j=1}^M \left(
      \modulo{\rho^{m+1}_{j,k_a} - \rho^m_{j,k_a}}
      + \modulo{\rho^{m+1}_{j,k_b+1} - \rho^m_{j,k_b+1}}
      \right)
    \\
    \leq \
    & e^{4 \, \mathcal{V} \, T} \left( \sum_{j=1}^M \tv (\rho_{o,j})
      +  8 \, M \, V_{\max}\, T
      + \frac{2 \, C}{s}\right),
  \end{align*}
  where we used also~\eqref{eq:20}.  Noticing that
  $[a,b] \subseteq [x_{k_a}, x_{k_b+1}]$ completes the proof.
\end{proof}

\subsection{Discrete Entropy Inequality}
\label{sec:entropyIneq}

We follow the idea of~\cite[Lemma~5.1]{KRT2002}.

\begin{lemma}
  \label{lem:entropyIneq}
  Let $\brho_o \in \BV (\reali; [0,1]^M)$ with
    $\NORM{\brho_o}< +\infty$. Assume that the CFL condition~\eqref{eq:cfl} holds. Then
  the approximate solution $\brho_\Delta$ defined by~\eqref{eq:6}
  through Algorithm~\ref{alg:1} satisfies the following discrete
  entropy inequality: for all $j=1,\dots,M$, for $k\in \interi$, for
  $n=0,\ldots, N_T-1$ and for any $c \in [0,1]$
  \begin{align}
    \nonumber
    \modulo{\rho^{n+1}_{j,k} - c} -
    \modulo{\rh{j,k} - c}
    + \lambda \left( \mathscr{F}^{c}_{j,k+1/2} (\rh{j,k}, \rh{j,k+1})
    -  \mathscr{F}^{c}_{j, k-1/2} (\rh{j,k-1}, \rh{j,k}) \right)
    &
    \\
    \label{eq:die}
    - \lambda \, \modulo{F_j (x_{k+1/2}, c,c) - F_j (x_{k-1/2},c, c)}
    &
    \\
    \nonumber
    - \dt \, \sgn (\rho^{n+1}_{j,k} - c)
    \left(S_{j-1} (x_{k}, \rho^{n+1/2}_{j-1,k}, \rho^{n+1/2}_{j,k})
    - S_j (x_k, \rho^{n+1/2}_{j,k}, \rho^{n+1/2}_{j+1,k})\right)
    &
      \leq \ 0,
  \end{align}
  with
  \begin{displaymath}
    \mathscr{F}^{c}_{j,k+1/2} (u,w)
    =
    F_j (x_{k+1/2},u \vee c, w \vee c)
    - F_j (x_{k+1/2},u \wedge c, w \wedge c),
  \end{displaymath}
  where $a \vee b = \max\{a,b\}$, $a \wedge b = \min\{a,b\}$.
\end{lemma}

\begin{proof}
  Fix $j \in \{1, \dots, M\}$ and $k \in \interi$. Let
  \begin{displaymath}
    \mathscr{G}_{j,k} (u,w,z) =
    w - \lambda \left[
      F_j (x_{k+1/2}, w, z)  - F_j (x_{k-1/2}, u,w)
    \right].
  \end{displaymath}
  Clearly
  $\rho^{n+1/2}_{j,k} = \mathscr{G}_{j,k} (\rh{j,k-1},
  \rh{j,k},\rh{j,k+1})$. Set
  \begin{displaymath}
    \Delta_k F_j ^c = F_j (x_{k+1/2}, c, c)  - F_j (x_{k-1/2}, c,c),
  \end{displaymath}
  so that
  $ \mathscr{G}_{j,k} (c,c,c) = c - \lambda \, \Delta_k F_j^c$.  By
  the properties of the numerical flux $F_j$, the map
  $\mathscr{G}_{j,k}$ is non decreasing in all its
  arguments. 
  Therefore,
  \begin{align*}
    \mathscr{G}_{j,k} (\rh{j,k-1} \vee c, \rh{j,k} \vee c, \rh{j,k+1} \vee c)
    \geq \
    & \mathscr{G}_{j,k} (\rh{j,k-1}, \rh{j,k}, \rh{j,k+1})  \vee \mathscr{G}_{j,k} (c,c,c),
    \\
    - \mathscr{G}_{j,k} (\rh{j,k-1} \wedge c, \rh{j,k} \wedge c, \rh{j,k+1} \wedge c)
    \geq \
    & - \mathscr{G}_{j,k} (\rh{j,k-1}, \rh{j,k}, \rh{j,k+1})
      \wedge \mathscr{G}_{j,k} (c,c,c).
  \end{align*}
  Sum the two inequalities above: since
  $a \vee b - a \wedge b = \modulo{a-b}$, observe that
  \begin{align*}
    & \mathscr{G}_{j,k} (\rh{j,k-1} \vee c, \rh{j,k} \vee c,
      \rh{j,k+1} \vee c) - \mathscr{G}_{j,k} (\rh{j,k-1} \wedge c,
      \rh{j,k} \wedge c, \rh{j,k+1} \wedge c)
      \\
      = \ & \modulo{\rho^{n}_{j,k} - c} -\lambda \left(
        \mathscr{F}^c_{j,k+1/2} (\rh{j,k}, \rh{j,k+1}) -
        \mathscr{F}^c_{j,k-1/2} (\rh{j,k-1}, \rh{j,k}) \right),
  \end{align*}
  and
  \begin{align*}
    & \mathscr{G}_{j,k} (\rh{j,k-1}, \rh{j,k}, \rh{j,k+1})  \vee \mathscr{G}_{j,k} (c,c,c)
      - \mathscr{G}_{j,k} (\rh{j,k-1}, \rh{j,k}, \rh{j,k+1})
      \wedge \mathscr{G}_{j,k} (c,c,c)
    \\
    = \
    & \modulo{\rho^{n+1/2}_{j,k} - c + \lambda \, \Delta_k F_j^c}
    \\
    = \
    & \modulo{\rho^{n+1}_{j,k} - c + \lambda \, \Delta_k F_j^c
      -\dt \left(
      S_{j-1} (x_k, \rho^{n+1/2}_{j-1,k}, \rho^{n+1/2}_{j,k})
      -
      S_{j} (x_k, \rho^{n+1/2}_{j,k}, \rho^{n+1/2}_{j+1,k})
      \right)
      }
    \\
    \geq \
    & \modulo{\rho^{n+1}_{j,k} - c}
      - \lambda \, \modulo{ \Delta_k F_j^c}
    \\
    & - \dt \, \sgn\!\left(\rho^{n+1}_{j,k} - c\right) \left(
      S_{j-1} (x_k, \rho^{n+1/2}_{j-1,k}, \rho^{n+1/2}_{j,k})
      -
      S_{j} (x_k, \rho^{n+1/2}_{j,k}, \rho^{n+1/2}_{j+1,k})
      \right),
  \end{align*}
  where we used also~\eqref{eq:scheme2} and the inequality
  $\modulo{a+b} \geq \modulo{a} + \sgn(a) \, b$.  The thesis
  immediately follows.
\end{proof}

\subsection{Convergence}
\label{sec:convergence}

\begin{theorem}
  \label{thm:main}
  Let $\brho_o \in \BV (\reali; [0,1]^M)$ with
  $\NORM{\brho_o}< +\infty$. Let $\dx \to 0$ with
  $\lambda = \dx / \dt$ constant and satisfying the CFL
  condition~\eqref{eq:cfl}. The sequence of approximate solutions
  $\brho_\Delta$ constructed through Algorithm~\ref{alg:1} converges
  in $\Lloc1$ to a function
  $\brho \in \L\infty([0,T] \times \reali; [0,1]^M)$ such that
  $\NORM{\brho(t)}= \NORM{\brho_o}$ for $t\in [0,T]$. This limit
  function $\brho$ is a weak entropy solution to
  problem~\eqref{eq:M}--\eqref{eq:idM}--\eqref{eq:idM1} in the sense
  of Definition~\ref{def:sol}.
\end{theorem}

\begin{proof}
  We follow~\cite[Theorem~5.1]{BKT2009}
  and~\cite[Theorem~5.1]{KRT2002}.

  Lemma~\ref{lem:pos} ensures that the sequence of approximate
  solutions $\brho_\Delta$ is bounded in $\L\infty$, in particular
  $\rho_{j,\Delta} (t,x) \in [0,1]$, for all $t>0$, $x \in \reali$ and
  $j=1,\dots,M$. Proposition~\ref{prop:l1cont-time} proves the
  $\L1$-continuity in time of the sequence $\brho_\Delta$, while
  Lemma~\ref{lem:bv} guarantees a bound on the spatial total variation
  in any interval $[a,b]$ not containing $x=0$. Standard compactness
  results imply that, for any interval $[a,b]$ not containing $x=0$,
  there exists a subsequence, still denoted by $\brho_\Delta$,
  converging in $\L1 ([0,T] \times [a,b]; [0,1]^M)$.

  Take now a countable set of intervals $[a_i,b_i]$ such that
  $\bigcup_i [a_i,b_i] = \reali \setminus \{0\}$: by a standard
  diagonal process, we can extract a subsequence, still denoted by
  $\brho_\Delta$, converging in
  $\Lloc1 ([0,T]\times \reali; [0,1]^M)$, and almost everywhere in
  $[0,T] \times \reali$, to a function
  $\brho \in \L\infty ([0,T]\times \reali; [0,1]^M)$. Moreover,
  Proposition~\ref{prop:l1cont-time}, and in particular
  formula~\eqref{eq:8}, implies that this limit function is such that
  $\brho \in \C0 ([0,T]; \L1 (\reali; [0,1]^M))$, with slight abuse of notation concerning the $\L1$-norm.

  It remains to show that the limit function $\brho$ satisfies the
  integral inequalities in Definition~\ref{def:sol}.  Concerning
  point~\ref{it:weak}, i.e.~the weak formulation, it suffices to apply
  a Lax-Wendroff-type calculation, similarly to what has been done
  in~\cite[Theorem~3.1]{KRT2002}. Notice that the presence of the
  source terms does not add any difficulties in the proof.

  As for point~\ref{it:entropy} in Definition~\ref{def:sol}, i.e.~the
  entropy inequality, we follow~\cite[Theorem~5.1]{KRT2002}. Fix
  $j \in \{1,\dots,M\}$.  Let
  $\phi \in \Cc1 ([0,T[ \times \reali; \reali^+)$. Multiply the
  inequality~\eqref{eq:die} by
  $\dx \, \phi_k^n = \dx \, \phi (t^n, x_k)$, then sum over
  $k \in \interi$ and $n=0, \ldots, N_T -1$:
  \begin{align}
    \label{eq:p1}
    0 \geq \
    & \dx \sum_{n=0}^{N_T -1} \sum_{k \in \interi}
      \left[
      \modulo{\rho^{n+1}_{j,k} - c} - \modulo{\rh{j,k} - c}
      \right] \phi^n_k
    \\
    \label{eq:p2}
    & + \dt \sum_{n=0}^{N_T -1} \sum_{k \in \interi}
      \left[
      \mathscr{F}^c_{j,k+1/2} (\rh{j,k}, \rh{j,k+1})
      - \mathscr{F}^c_{j,k-1/2} (\rh{j,k-1}, \rh{j,k})
      \right] \phi^n_k
    \\
    \label{eq:p3}
    & - \dt \sum_{n=0}^{N_T -1} \sum_{k \in \interi}
      \modulo{ F_j (x_{k+1/2}, c, c) - F_j (x_{k-1/2}, c, c)} \, \phi^n_k
    \\
    \label{eq:p4}
    & -  \dt \, \dx \sum_{n=0}^{N_T -1} \sum_{k \in \interi}
      \sgn (\rho^{n+1}_{j,k} - c)
      \left[
      S_{j-1} (x_k, \rho^{n+1/2}_{j-1,k}, \rho^{n+1/2}_{j,k})
      - S_j (x_k, \rho^{n+1/2}_{j,k}, \rho^{n+1/2}_{j+1,k})
      \right] \phi^n_k.
  \end{align}
  Take into account each term separately. Summing by parts and letting
  $\dx \to 0^+$, the Dominated Convergence Theorem yields
  \begin{equation}
    \label{eq:p1ok}
    \begin{aligned}
      [\eqref{eq:p1}] = \ & - \dx \sum_{k \in \interi}
      \modulo{\rho^0_{j,k} - c} \, \phi^0_k - \dx \, \dt
      \sum_{n=1}^{N_T -1} \sum_{k \in \interi} \modulo{\rh{j,k} - c}
      \, \frac{\phi^n_k - \phi^{n-1}_k}{\dt}
      \\
      \underset{\dx \to 0^+}{\longrightarrow} \ & - \int_{\reali}
      \modulo{\rho_{o,j} - c} \, \phi (0,x) \d{x} - \int_0^T
      \int_\reali \modulo{\rho_j (t,x) - c} \, \partial_t \phi (t,x)
      \d{x}\d{t},
    \end{aligned}
  \end{equation}
  and
  \begin{equation}
    \label{eq:p2ok}
    \begin{aligned}
      [\eqref{eq:p2}] = \ & - \dx \, \dt \sum_{n=0}^{N_T -1} \sum_{k
        \in \interi} \mathscr{F}^c_{j,k+1/2} (\rh{j,k}, \rh{j,k+1})
      \frac{\phi^n_k - \phi^n_{k-1}}{\dx}
      \\
      \underset{\dx \to 0^+}{\longrightarrow} \ & - \int_0^T
      \int_\reali \sgn(\rho_j (t,x) - c) \left(f_j (x,\rho_j (t,x)) -
        f_j (x,c)\right) \, \partial_x \phi (t,x) \d{x}\d{t}.
    \end{aligned}
  \end{equation}
  Pass now to~\eqref{eq:p3}. Observe that, by the definition of the
  numerical flux~\eqref{eq:fx}, when $x \neq 0$ it holds
  $F_j (x,c,c) = f_{d,j} (c)$, with $d=\ell$ if $x <0$ and $d=r$ if
  $x>0$. Therefore~\eqref{eq:p3} gives a contribution only for $k=-1$
  and $k=0$:
  \begin{align*}
    [\eqref{eq:p3}] =\
    & - \dt \sum_{n=0}^{N_T -1} \sum_{k=-1}^0
      \modulo{F_j (x_{k+1/2},c,c) - F_j (x_{k-1/2}, c, c)} \, \phi_{k}^n
    \\
    \underset{\dx \to 0^+}{\longrightarrow} \
    & - \int_0^T\left(\modulo{F_j (0,c,c) - f_{\ell,j} (c)}
      + \modulo{f_{r,j} (c) - F_j (0,c,c)}\right) \, \phi (t,0) \d{t}
  \end{align*}
  A careful analysis of all the possible cases yields
  \begin{displaymath}
    \modulo{F_j (0,c,c) - f_{\ell,j} (c)} + \modulo{f_{r,j} (c) - F_j (0,c,c)}
    = \modulo{f_{r,j} (c) - f_{\ell,j} (c)},
  \end{displaymath}
  so that
  \begin{equation}
    \label{eq:p3ok}
    [\eqref{eq:p3}]   \underset{\dx \to 0^+}{\longrightarrow}
    - \int_0^T \modulo{f_{r,j} (c) - f_{\ell,j} (c)} \, \phi (t,0) \d{t}.
  \end{equation}
  Focus now on the last term~\eqref{eq:p4}: by the Dominated
  Convergence Theorem
  \begin{equation}
    \label{eq:p4ok}
    [\eqref{eq:p4}]  \underset{\dx \to 0^+}{\longrightarrow}
    - \int_0^T \int_\reali \sgn(\rho_j- c)
    \left(
      S_{j-1} (x, \rho_{j-1} , \rho_j ) - S_j (x, \rho_j , \rho_{j+1} )
    \right) \phi (t,x) \d{x} \d{t}.
  \end{equation}
  Collecting together~\eqref{eq:p1ok}, \eqref{eq:p2ok},
  \eqref{eq:p3ok} and~\eqref{eq:p4ok} completes the proof.
\end{proof}

\subsection{\texorpdfstring{$\L1$}{L1}-Stability and uniqueness}
\label{sec:unique}

The following Theorem ensures that the solution
to~\eqref{eq:M}--\eqref{eq:idM}--\eqref{eq:idM1}
depends $\L1$-Lipschitz continuously on the initial data, thus
guaranteeing the uniqueness of solutions.

\begin{theorem}
  \label{thm:unique}
  Let $\brho, \, \bsigma$ be two weak entropy solutions, in the sense
  of Definition~\ref{def:sol}, to
  problem~\eqref{eq:M}--\eqref{eq:idM}--\eqref{eq:idM1}
  with initial data
  $\brho_o, \, \bsigma_o \in \L\infty (\reali;
    [0,1]^M)$ and such that
    $\brho_o - \bsigma_o \in \L1 (\reali;[0,1]^M)$.
  Then, for a.e.~$t \in [0,T]$,
  \begin{equation}
    \label{eq:22}
    \sum_{j=1}^M \norma{\rho_j (t) - \sigma_j (t)}_{\L1 (\reali)}
    \leq
    \sum_{j=1}^M \norma{\rho_{o,j} - \sigma_{o,j}}_{\L1 (\reali)}.
  \end{equation}
\end{theorem}
\begin{remark}
  Notice that the sums appearing in~\eqref{eq:22} are actually sums
  over the \emph{active} lanes only, the terms corresponding to
  fictive lanes being equal to $0$.
\end{remark}

\begin{proof}
  The idea is to combine together the results contained in~\cite[\S~2
  and \S~5]{KRT2003}, in particular~\cite[Theorem~5.1]{KRT2003}, and
  in~\cite[Theorem~3.1]{BKT2009}, and then
  adapt~\cite[Theorem~3.3]{HoldenRisebro}.

  Indeed, fix $j\in\{1, \dots, M\}$. Following~\cite[Theorem~A.1 and
  Formula~(2.22)]{KRT2003}, it is possible to derive the following
  inequality for any
  $\phi \in \Cc1 (\,]0,T[ \times \reali\setminus\{0\}; \reali^+)$
  \begin{equation}
    \label{eq:1}
    \begin{aligned}
      -\int_0^T \int_\reali \bigl\{ \modulo{\rho_j - \sigma_j}
      \partial_t \phi + \sgn(\rho_j - \sigma_j) \left( f_j (x,\rho_j)
        - f_j (x,\sigma_j)\right) \partial_x \phi &
      \\
      + \sgn(\rho_j - \sigma_j)\left(S (x, \brho, j) - S (x, \bsigma,
        j)\right) \phi & \bigr\} \d{x}\d{t} \leq 0,
    \end{aligned}
  \end{equation}
  where, for the sake of simplicity, we set
  \begin{equation}
    \label{eq:S}
    S (x, \boldsymbol{u}, j) =  S_{j-1} (x,u_{j-1}, u_j) - S_j (x, u_j, u_{j+1}).
  \end{equation}
  Inspired by~\cite[Theorem~3.3]{HoldenRisebro}, since $\rho_j$,
  respectively $\sigma_j$, satisfies Point~\ref{it:weak} in
  Definition~\ref{def:sol}, we subtract to the above inequality the
  equation for $\rho_j$ and add the equation for $\sigma_j$, arriving
  at
  \begin{align*}
    -\int_0^T\int_\reali \bigl\{\left(\rho_j - \sigma_j\right)^+\partial_t \phi
    + H (\rho_j - \sigma_j) \left(f_j (x,\rho_j) - f_j (x, \sigma_j)\right) \partial_x\phi
    &
    \\
    + H (\rho_j - \sigma_j) \left(S (x,\brho,j) - S(x,\bsigma,j)\right)\phi
    & \bigr\} \d{x} \d{t}\leq 0,
  \end{align*}
  for $\phi \in \Cc1 (\,]0,T[ \times \reali\setminus\{0\};
  \reali^+)$. Now, we extend the above inequality to
  $\Phi \in \Cc1 (\,]0,T[ \times \reali; \reali^+)$. The procedure is
  similar to that in~\cite[Theorem~2.1]{KRT2003} and it leads to
  \begin{equation}
    \label{eq:5}
    \begin{aligned}
      -\int_0^T\int_\reali \bigl\{\left(\rho_j - \sigma_j\right)^+
      \partial_t \Phi + H (\rho_j - \sigma_j) \left(f_j (x,\rho_j) -
        f_j (x, \sigma_j)\right) \partial_x \Phi &
      \\
      + H (\rho_j - \sigma_j) \left(S (x,\brho,j) -
        S(x,\bsigma,j)\right) \Phi & \bigr\} \d{x} \d{t} \leq E,
    \end{aligned}
  \end{equation}
  for all $\Phi \in \Cc1 (\,]0,T[ \times \reali; \reali^+)$, where
  \begin{displaymath}
    E = \int_0^T\left[  H(\rho_j - \sigma_j) \left( f_j (x,\rho_j) - f_j (x,\sigma_j)\right)
    \right]_{x=0^-}^{x=0^+} \Phi (t,0) \d{t}.
  \end{displaymath}
  Analogously to~\cite[Theorem~2.1]{KRT2003}
  and~\cite[Theorem~3.1]{BKT2009}, it can be proven that $E\leq 0$.
  Following again~\cite[Theorem~3.3]{HoldenRisebro} and choosing
  $\Phi \approx \boldsymbol{1}_{[0,\tau]}$, for $\tau \in \,]0,T]$, we
  get
  \begin{equation}
    \label{eq:18}
    \begin{aligned}
      \int_\reali\left(\rho_j (\tau) - \sigma_j (\tau)\right)^+ \d{x}
      \leq \ & \int_\reali\left(\rho_j (0) - \sigma_j
        (0)\right)^+\d{x}
      \\
      & + \int_0^\tau \int_\reali H (\rho_j - \sigma_j) \left(S
        (x,\brho,j) - S(x,\bsigma,j)\right) \d{x} \d{t}.
    \end{aligned}
  \end{equation}
  It is easy to verify that, for fixed $x$, the map $S_j (x,u,w)$
  defined in~\eqref{eq:4}, together with~\eqref{eq:2}, is non
  decreasing in the second argument and non increasing in the third:
  setting for the sake of convenience
  $\Delta_+ v_j =v_{j+1} (x,w) - v_{j} (x,u)$ we obtain
  \begin{align*}
    \partial_u S_j  = \
    & (\Delta_+v_j )^+
      -  v'_j (x,u) \, w
      - H (\Delta_+ v_j)
      \, v'_j (x,u) \, (u-w)
      \geq 0,
    \\
    \partial_w S_j = \
    & - (\Delta_+v_j)^- 
      +  v'_{j+1} (x,w) \, w
      + H (\Delta_+v_j) 
      \, v'_{j+1} (x,w) \, (u-w)
      \leq 0.
  \end{align*}
  Hence, if $\rho_j > \sigma_j$ we have
  \begin{align*}
    & S (x,\brho,j) - S(x,\bsigma,j)
    \\ =  \
    & S_{j-1} (x,\rho_{j-1}, \rho_j)-S_{j-1} (x,\sigma_{j-1}, \sigma_j)
      - S_j (x, \rho_j, \rho_{j+1}) + S_j (x, \sigma_j, \sigma_{j+1})
    \\
    \leq \
    &  S_{j-1} (x,\rho_{j-1}, \sigma_j)-S_{j-1} (x,\sigma_{j-1}, \sigma_j)
      - S_j (x, \rho_j, \rho_{j+1}) + S_j (x, \rho_j, \sigma_{j+1})
    \\
    = \
    & \partial_u S_{j-1} (x, \pi_{j-1}, \sigma_j)\, (\rho_{j-1} - \sigma_{j-1}) -
      \partial_w S_{j} (x,\sigma_j, \pi_{j+1}) \,(\rho_{j+1}- \sigma_{j+1})
    \\
    \leq \
    & \mathcal{V}
      \left((\rho_{j-1} - \sigma_{j-1})^+ + (\rho_{j+1}- \sigma_{j+1})^+\right),
  \end{align*}
  with $\pi_{j\pm1}$ in the interval between $\rho_{j\pm1}$ and
  $\sigma_{j\pm1}$ respectively and $\mathcal{V}$ as
  in~\eqref{eq:normav}. Thus,
  \begin{equation}
    \label{eq:25}
    \sum_{j=1}^M
    H (\rho_j - \sigma_j) \left(S (x,\brho,j) - S(x,\bsigma,j)\right)
    \leq 2 \, \mathcal{V} \,  \sum_{j=1}^M (\rho_j - \sigma_j)^+.
  \end{equation}
  Define
  \begin{displaymath}
    \Theta (t) = \sum_{j=1}^M \int_\reali \left(\rho_j (t,x) - \sigma_j (t,x)\right)^+ \d{x}.
  \end{displaymath}
  By~\eqref{eq:18} and~\eqref{eq:25} it follows that
  \begin{displaymath}
    \Theta (\tau) \leq \Theta (0) + 2 \, \mathcal{V} \int_0^\tau \Theta (t)  \d{t}.
  \end{displaymath}
  Gronwall's inequality then implies that
  $\Theta (t) \leq \Theta (0) \, \exp \left( 2 \, \mathcal{V} \,
    t\right)$. Therefore, if $\Theta (0)=0$,
  i.e.~$\rho_{o,j}(x)\leq \sigma_{o,j} (x)$ a.e.~in $\reali$ and for
  all $j$, then $\Theta (t)=0$ for $t>0$,
  i.e.~$\rho_{j}(t,x)\leq \sigma_{j} (t,x)$ a.e.~in $\reali$ and for
  all $j$. An application of the Crandall--Tartar
  Lemma~\cite[Lemma~2.13]{HoldenRisebroBook2015} concludes the proof
  of the $\L1$--contractivity.
\end{proof}

\section{Numerical experiments}
\label{sec:num}

We present some applications of our result in test cases describing
realistic road junction examples. The study is not exhaustive: in
particular, specific cases of diverging junctions could be handled
adding some information on drivers' routing preferences upstream the
junction. Yet, these situations go beyond the scope of this paper.

In all the numerical experiments, we choose
\begin{displaymath}
  v_{d,j} (u) = V_d (1-u)
  \quad
  \hbox{for}~d=\ell, r~ \hbox{and}~j=1,\ldots,M,
\end{displaymath}
thus the maximal speed is the same for all the lanes before,
respectively after, $x=0$. In particular, in each situation we
consider two cases, $V_\ell < V_r$ and $V_\ell > V_r$.

\subsection{1-to-1 junction: from 2 to 3 lanes}
\label{sec:incr}

We consider problem~\eqref{eq:M}--\eqref{eq:idM}--\eqref{eq:sghost},
with $\M_\ell=\left\{1,2\right\}$, $\M_r=\left\{1,2,3\right\}$ and
$S_{\ell,2}(u,w)=0$.
\begin{center}
  \begin{tikzpicture}[yscale=0.5]
    \draw[ultra thick] (0,0) -- (6,0) node[pos=0.25, align=center, above]{lane 1}
    node[pos=0.5, align=center, below] {$x=0$};
    \draw[ultra thick, dashed] (0,1) -- (6,1) node[pos=0.25, align=center, above]{lane 2};
    \draw[ultra thick] (0,2) -- (3,2) (3,2) |- (6,3);
    \draw[ultra thick, dashed] (3,2) -- (6,2) node[pos=0.5,
    align=center, above]{lane 3}; \draw[thick, dotted] (3,0) -- (3,2);
  \end{tikzpicture}
\end{center}
The initial data are chosen as follows:
\begin{align}
  \label{eq:27}
  \rho_{o,1} (x) = \
  & 0.7,
  &
    \rho_{o,2} (x) = \
  & 0.6,
  &
    \rho_{o,3} (x) = \
  & 0.5 * \caratt{[0,+\infty[} (x).
\end{align}
Moreover, we choose $ V_\ell = 1.5$, and $V_{r}=1$ or $2$
respectively.  Figure~\ref{fig:2to3} displays the solutions in both
cases at time $t= 1$: on the right the maximal speed decreases, on
the left it is increasing. We notice the effect of the flow between
neighbouring lanes: all along the $x$-axis vehicles moves from lane 1
to lane 2, for $x>0$ vehicles pass also from lane 2 to lane 3, and
this is particularly evident near $x=0$.
\begin{figure}[!h]
  \centering \includegraphics[width=0.48\textwidth, trim=20 5 35 10,
  clip=true]{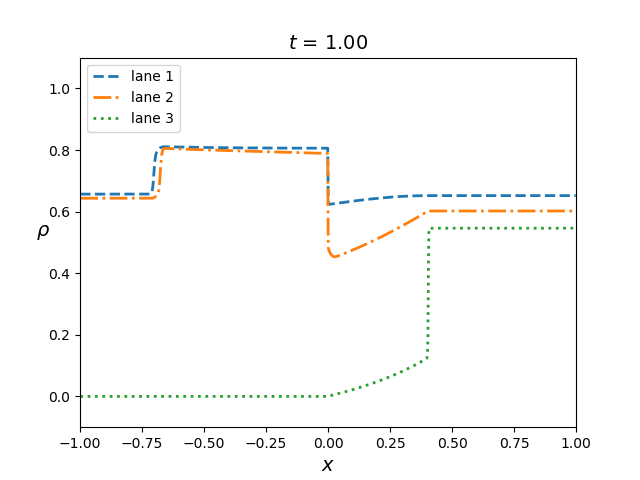}\hspace{8pt}%
  \includegraphics[width=0.48\textwidth, trim=20 5 35 10, clip=true]{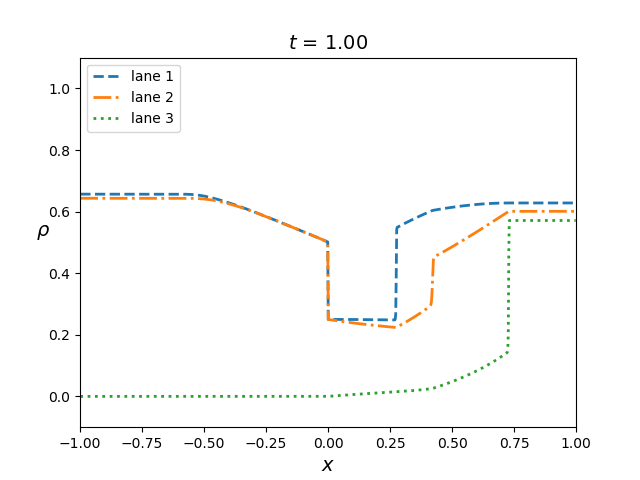}\\
  \caption{Solutions
    to~\eqref{eq:M}--\eqref{eq:idM}--\eqref{eq:sghost}, with
    $\M_\ell=\left\{1,2\right\}$, $\M_r=\left\{1,2,3\right\}$ and
    initial data~\eqref{eq:27} at time $t=1$.  $V_\ell = 1.5$: left
    $V_r=1$, right $V_r=2$.}
  \label{fig:2to3}
\end{figure}

\subsection{1-to-1 junction: from 3 to 2 lanes}
\label{sec:decr2}

We consider problem~\eqref{eq:M}--\eqref{eq:idM1}--\eqref{eq:sghost},
with $\M_\ell=\left\{1,2,3\right\}$, $\M_r=\left\{1,2\right\}$ and
$S_{r,2}(u,w)=0$.
\begin{center}
  \begin{tikzpicture}[yscale=0.5]
    \draw[ultra thick] (0,0) -- (6,0) node[pos=0.25, align=center,
    above]{lane 1} node[pos=0.5, align=center, below] {$x=0$};
    \draw[ultra thick, dashed] (0,1) -- (6,1) node[pos=0.25,
    align=center, above]{lane 2}; \draw[ultra thick] (0,3) -| (3,2)
    (3,2) -- (6,2); \draw[ultra thick, dashed] (0,2) -- (3,2)
    node[pos=0.5, align=center, above]{lane 3}; \draw[thick, dotted]
    (3,0) -- (3,2);
  \end{tikzpicture}
\end{center}
The initial data are chosen as follows:
\begin{align}
  \label{eq:26}
  \rho_{o,1} (x) = \
  & 0.7,
  &
    \rho_{o,2} (x) = \
  & 0.6,
  &
    \rho_{o,3} (x) = \
  & 0.5 \, \caratt{]-\infty,0]} (x)+ 1 \,\caratt{]0,+\infty[} (x).
\end{align}
We choose $ V_\ell = 1.5$, and $V_{r}=1$ or $2$ respectively.
Figure~\ref{fig:3to2} displays the solutions in both cases at time
$t= 1$: on the right the maximal speed decreases, on the left it is
increasing. We display the solution also for the positive part of the
third lane: it is constantly equal to the maximal density $1$. As in
the case of an increasing number of lanes, we notice the effect of the
flow between neighbouring lanes. Observe that no vehicle passes from
lane 3 to lane 2 for $x>0$: indeed, lane 3 for $x>0$ is a fictive lane
and we impose~\eqref{eq:sghost} ($\sr{2} (u,w)=0$).
\begin{figure}[!h]
  \includegraphics[width=0.48\textwidth, trim=20 5 35 10,
  clip=true]{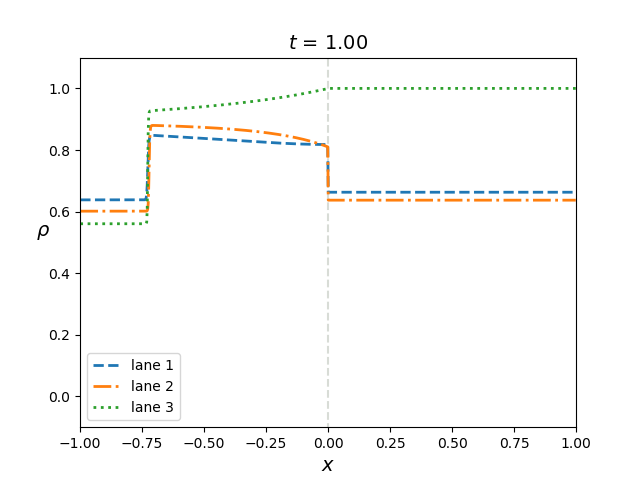}\hspace{8pt}%
  \includegraphics[width=0.48\textwidth, trim=20 5 35 10,
  clip=true]{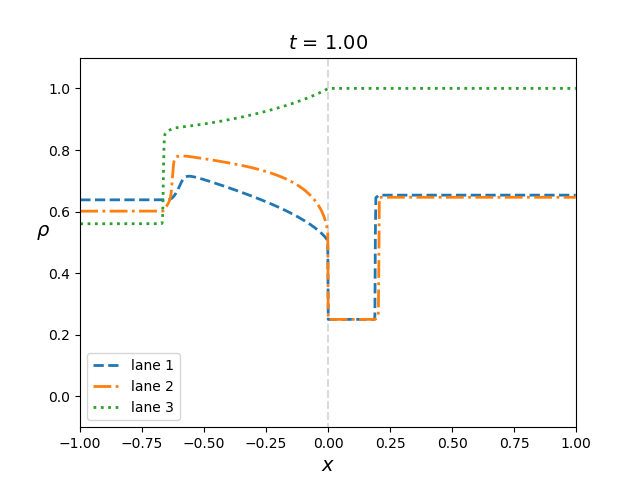}
  \caption{Solutions
    to~\eqref{eq:M}--\eqref{eq:idM1}--\eqref{eq:sghost}, with
    $\M_\ell=\left\{1,2,3\right\}$, $\M_r=\left\{1,2\right\}$ and
    initial data~\eqref{eq:26} at time $t=1$.  $V_\ell = 1.5$: left
    $V_r=1$, right $V_r=2$.}
  \label{fig:3to2}
\end{figure}

Focus on the queue forming before $x=0$ and compare the two cases,
$V_r<V_\ell$ and $V_r>V_\ell$. When the maximal speed diminishes, the
queue is longer and the number of vehicles in the queue
is greater with respect to the case of increasing maximal speed: for
$x<0$, in the former case it is more difficult for vehicles in lane 3
to pass in lane 2, since here the decrease in the maximal speed
diminishes the flow at $x=0$.

\subsection{2-to-1 junction : from 3 to 2 lanes}
\label{sec:32plud}

We consider the same setting of Section~\ref{sec:decr2}, thus
problem~\eqref{eq:M}--\eqref{eq:idM1}--\eqref{eq:sghost}, with
$\M_\ell=\left\{1,2,3\right\}$, $\M_r=\left\{1,2\right\}$ and initial
data~\eqref{eq:26}, with the additional assumption that there is no
flow of vehicles between the first and the second lane on
$]-\infty,0[$, i.e.~$\se{1} (u,w)=0$ (we keep $S_{r,2}(u,w)=0$):
\begin{center}
  \begin{tikzpicture}[yscale=0.5]
    \draw[ultra thick] (0,0) -- (6,0) node[pos=0.25, align=center,
    above]{lane 1} node[pos=0.5, align=center, below] {$x=0$};
    \draw[ultra thick] (0,1) -- (3,1) node[pos=0.5, align=center,
    above]{lane 2}; \draw[ultra thick] (0,3) -| (3,2) (3,2) -- (6,2);
    \draw[ultra thick, dashed] (3,1) --(6,1) (0,2) -- (3,2)
    node[pos=0.5, align=center, above]{lane 3}; \draw[thick, dotted]
    (3,0) -- (3,2);
  \end{tikzpicture}
\end{center}
We choose $V_\ell = 1.5$ and $V_r \in \{1, \, 1.5, \,
2\}$. Figure~\ref{fig:3to2Plus} displays the solution in the three
cases at time $t=0.5$: on the right the maximal speed decreases, in
the centre it stays constant, on the left it increases. As before, we
display the solution also for the positive part of the third lane,
where it is constantly equal to $1$.
\begin{figure}[!h]
  \centering \includegraphics[width=0.31\textwidth, trim=20 5 35 10,
  clip=true]{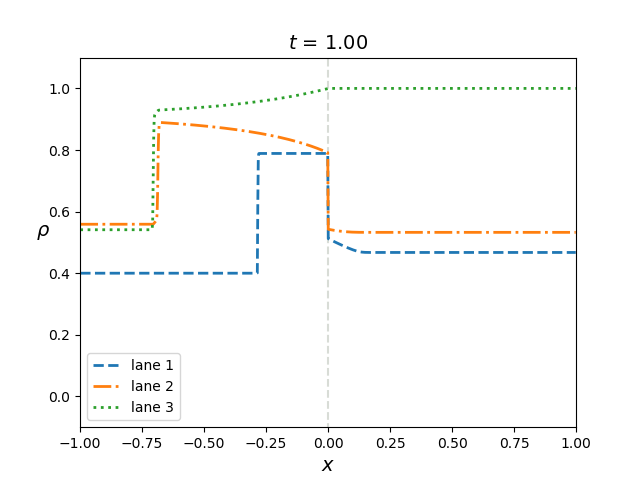}\hspace{8pt}%
  \includegraphics[width=0.31\textwidth, trim=20 5 35 10,
  clip=true]{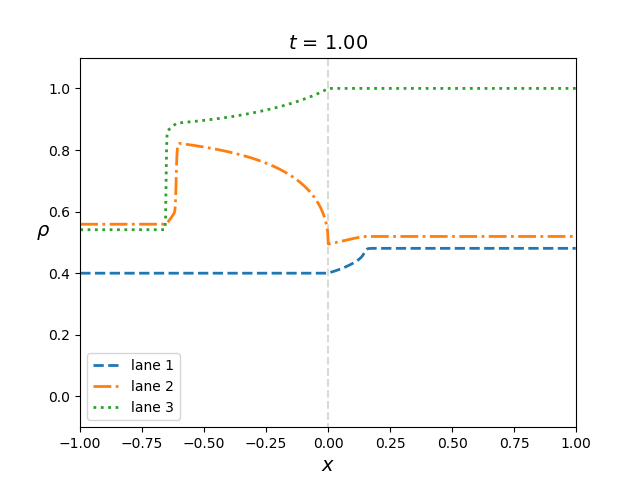}\hspace{8pt}%
  \includegraphics[width=0.31\textwidth, trim=20 5 35 10,
  clip=true]{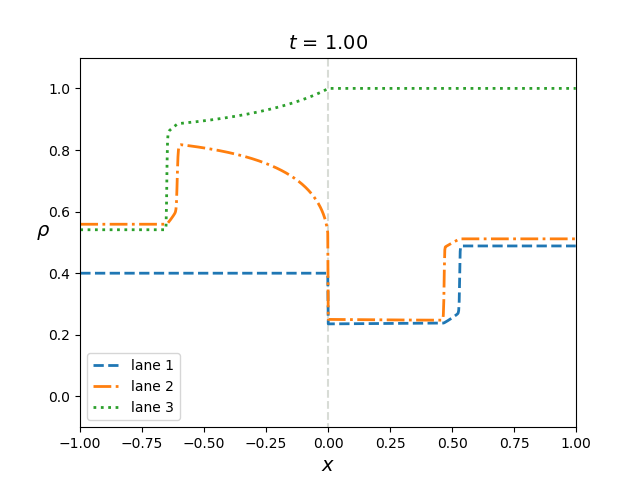}
  \caption{Solutions
    to~\eqref{eq:M}--\eqref{eq:idM1}--\eqref{eq:sghost} and
    $\se{1} (u,w)=0$, with $M_\ell=3$, $M_r=2$ and initial
    data~\eqref{eq:26} at time $t=1$.  $V_\ell = 1.5$: left
    $V_r=1$, centre $V_r=1.5$, right $V_r=2$.}
  \label{fig:3to2Plus}
\end{figure}

\subsection{2-to-1 junction: from 4 to 2 lanes}
\label{sec:42}

We consider the
problem~\eqref{eq:M}--\eqref{eq:idM1}--\eqref{eq:sghost}, with
$\M_\ell=\left\{1,2,3,4 \right\}$, $\M_r=\left\{2,3\right\}$ and
initial data
\begin{equation}
  \label{eq:28}
  \begin{aligned}
    \rho_{o,1} (x) = \ & 0.7 \, \caratt{]-\infty,0]} (x)+ 1
    \,\caratt{]0,+\infty[} (x), & \rho_{o,2} (x) = \ & 0.5,
    \\
    \rho_{o,3} (x) = \ & 0.6, & \rho_{o,4} (x) = \ & 0.4 \,
    \caratt{]-\infty,0]} (x)+ 1 \,\caratt{]0,+\infty[} (x),
  \end{aligned}
\end{equation}
with the additional assumption that there is no flow of vehicles
between the second and the third lane on $]-\infty,0[$,
i.e.~$\se{2} (u,w)=0$ (we also impose $S_{r,1}(u,w)=S_{r,3}(u,w)=0$).
The situation under consideration looks as follows:
\begin{center}
  \begin{tikzpicture}[yscale=0.5]
    \draw[ultra thick] (0,0) -| (3,1) (3,1) -- (6,1) (0,2) -- (3,2)
    node[pos=0.5, align=center, below]{lane 2} node[pos=0.5,
    align=center, above]{lane 3} (3,3) -- (6,3) (0,4) -| (3,3);
    \draw[ultra thick, dashed] (0,1) -- (3,1) node[pos=0.5,
    align=center, below]{lane 1} (3,2) --(6,2) (0,3) -- (3,3)
    node[pos=0.5, align=center, above]{lane 4}; \draw[thick, dotted]
    (3,1) -- (3,3); \draw (3,0) node[below] {$x=0$};
  \end{tikzpicture}
\end{center}
We choose $V_\ell = 1.5$ and $V_r \in \{1, \, 1.5, \,
2\}$. Figure~\ref{fig:4to2} displays the solution in the three cases
at time $t=1$: on the right the maximal speed decreases, in the
centre it stays constant, on the left it increases. As before, we
display the solution also for the positive part of the first and
fourth lane, where it is constantly $1$.
\begin{figure}[!h]
  \centering \includegraphics[width=0.31\textwidth, trim=20 5 35 10,
  clip=true]{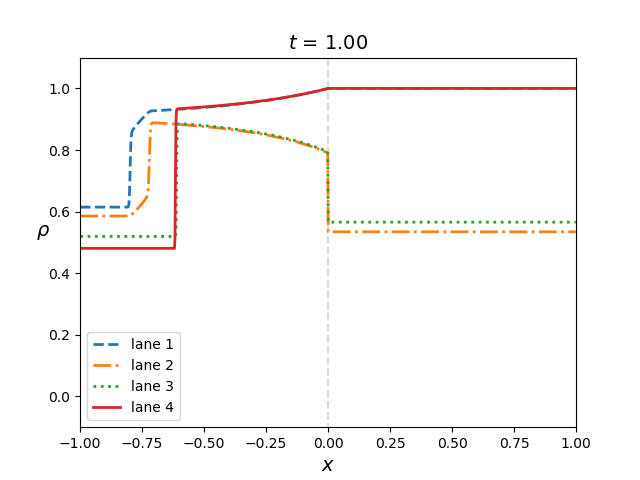}\hspace{8pt}%
  \includegraphics[width=0.31\textwidth, trim=20 5 35 10,
  clip=true]{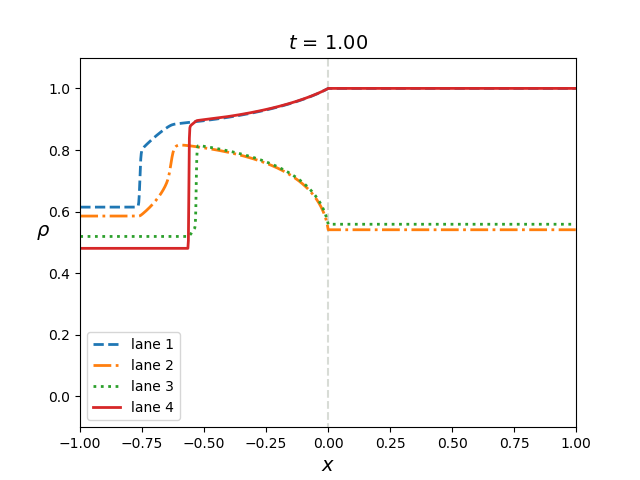}\hspace{8pt}%
  \includegraphics[width=0.31\textwidth, trim=20 5 35 10,
  clip=true]{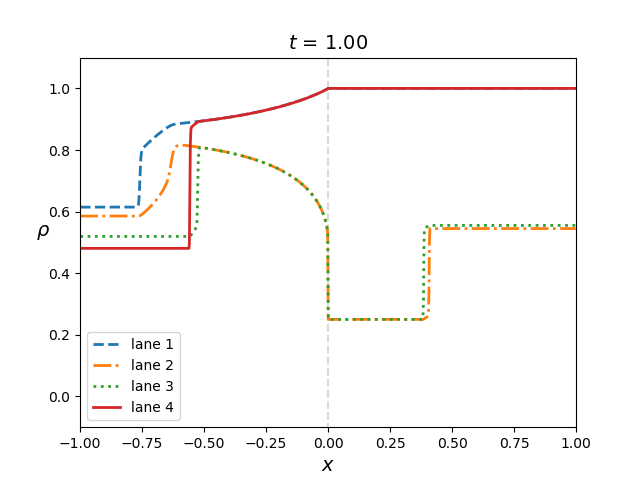}
  \caption{Solutions
    to~\eqref{eq:M}--\eqref{eq:idM1}--\eqref{eq:sghost} and
    $\se{21} (u,w)=0$, with $\M_\ell=\left\{1,2,3,4 \right\}$,
    $\M_r=\left\{2,3\right\}$ and initial data~\eqref{eq:28} at time
    $t=1$.  $V_\ell = 1.5$: left $V_r=1$, centre $V_r=1.5$, right
    $V_r=2$.}
  \label{fig:4to2}
\end{figure}

\bigskip

\noindent\textbf{Acknowledgement:}
The authors are grateful to Rinaldo M.~Colombo for stimulating
discussions.

{ \small

  \bibliography{multilane}

  \bibliographystyle{abbrv}

}

\end{document}